\documentclass[12pt,a4paper,reqno]{amsart}

\usepackage[margin=2cm]{geometry}
\usepackage{amssymb}
\usepackage{amsmath,amsthm}
\usepackage{amsfonts}
\usepackage{mathrsfs}
\usepackage[foot]{amsaddr}
\usepackage{color}
\usepackage[colorlinks=true,urlcolor=blue,
citecolor=red,linkcolor=blue,linktocpage,pdfpagelabels,
bookmarksnumbered,bookmarksopen]{hyperref}

\newcommand{\R}{{\mathbb R}}

\newcommand{\ds}{\displaystyle}

\newtheorem{theor}{Theorem}[section]

\newtheorem{lem}{Lemma}[section]
\newtheorem{prop}{Proposition}[section]
\newtheorem{cor}{Corollary}[section]

\newtheorem{rem}{Remark}[section]
\title[On a critical Kirchhoff type problem in high dimension]{ On the Brezis-Nirenberg problem for a Kirchhoff type equation  in high dimension}

\author{F. Faraci}
\address[F. Faraci]{Department of Mathematics and Computer Sciences, University of Catania, 95125 Catania, Italy}
\email{ffaraci@dmi.unict.it}

\author{K. Silva}
\address[K. Silva]{Instituto de Matem\'{a}tica e Estat\'{i}stica, Universidade Federal de Goi\'{a}s, Goi\^{a}nia GO74001-970,
Brazil}
\email{kayesilva@ufg.br}
\thanks{The second author is the corresponding author.}

\makeatletter
\newcommand{\mylabel}[2]{#2\def\@currentlabel{#2}\label{#1}}
\makeatother

\begin{document}

	\begin{abstract} The present paper  deals with a parametrized Kirchhoff type problem involving a critical nonlinearity in high dimension. Existence, non existence and multiplicity of solutions are obtained under the effect of a subcritical perturbation by combining variational properties  with a careful analysis of the fiber maps of the energy functional associated to the problem. The particular case of a pure power perturbation is also addressed. Through the study of the Nehari manifolds we extend the general case to  a wider range of the parameters. 
		
	\end{abstract}

\maketitle

\begin{center}
	\begin{minipage}{12cm}
		\tableofcontents
	\end{minipage}
\end{center}

 \bigskip
 

\smallskip


\section{Introduction and main results}	
Nonlocal boundary value problems of the type
$$
\left\{
\begin{array}{ll}
- \left( a+b\ds\int_\Omega |\nabla u|^2 dx\right)\Delta u=
f(x,u), & \hbox{ in } \Omega \\ \\
u=0, & \hbox{on } \partial \Omega
\end{array}
\right.
$$
are related to the stationary version of the equation
$$\frac{\partial^2 u}{\partial t^2}- \left( a+b\ds\int_\Omega |\nabla u|^2 dx\right)\Delta u=f(t,x,u),$$  proposed by Kirchhoff (\cite{K}) as a generalization of the D'Alembert's wave  equation  to describe the transversal oscillations of a stretched string. Here  $u$ denotes the displacement, $f$ is the
external force, $b$ is the initial tension and $a$ is related to the intrinsic properties
of the string. The importance of these kind of problems and its mathematical developments were made very clear on the recent short survey \cite{PR}.  

Recently, the existence and multiplicity of solutions of  Kirchhoff problems   under the effect of a critical nonlinearity $f$ have received considerable attention. Indeed, the challenging feature of such  problems is due to the presence of a nonlocal term together with the lack of compactness of the Sobolev embedding $H^{1}_0(\Omega)\hookrightarrow L^{2^\star}(\Omega)$ which prevent the application of standard variational methods.

The existence and multiplicity of solutions of  Kirchhoff type equations  with critical exponents  have been investigated by using different techniques as truncation and variational methods,  the Nehari manifold approach, the Ljusternik--Schnirelmann category theory,  genus theory (see for instance \cite{ CF,Fan, FS} and the references therein).

In the recent works  \cite{ACF, F, H1, H2, N1, DN,  YM},  an application of the Lions' Concentration Compactness principle  allows to prove  the Palais Smale condition of the energy functional, a key property  for the application of the well known Mountain Pass Theorem. Notice that according to the space dimension $N$, the geometry of the energy functional changes  and when $N\geq 4$ (coercive case) the property holds when $a$ and $b$ satisfy a suitable constraint (see \cite{H1,H2, DN, YM}).

Indeed, when $N\geq 4$, in \cite{FF} it is shown that the interaction between the Kirchhoff operator and the critical term leads to some useful variational properties of the energy functional such as the weak lower semicontinuity and the Palais Smale property when $a^\frac{N-4}{2}b\geq  C_1(N)$ or $a^\frac{N-4}{2}b> C_2(N)$ respectively, for suitable constants $C_1(N)<C_2(N)$.

In this paper we study  the following critical Kirchhoff problem
$$
(\mathcal P_\lambda) \ \ \ \ \left\{
  \begin{array}{ll}
    - \left( a+b\ds\int_\Omega |\nabla u|^2 dx\right)\Delta u=
|u|^{2^*-2}u+\lambda f(x,u), & \hbox{ in } \Omega \\ \\
    u=0, & \hbox{on } \partial \Omega
  \end{array}
\right.
$$
where  $\Omega\subseteq \R^N$ ($N>4$) is a bounded domain,   $a, b$ are positive fixed numbers, $2^*$ is the Sobolev  critical exponent,
 $\lambda$ is a positive  parameter, $f$ a subcritical Carath\'{e}odory function. 

In the present paper, through a careful analysis of the fiber maps associated to the energy functional, we will study the existence, non existence and the multiplicity of solutions of $(\mathcal P_\lambda)$. Indeed, by using the fibration method introduced in \cite{P} and the notion of extremal values of \cite{Y}, we will describe the topological changes of the energy functional, when  the parameters $a,b,\lambda$ vary. As it will become clear throughout our study, from the very geometry of the fibers, we will be able to deduce  a precise, and in some cases complete picture on existence, non-existence and multiplicity results.
	
	When the nonlinearity $f$ is a pure power term, i.e. $f(x,u)=|u|^{p-2}u$ for some $p\in (2,2^*)$, we will  go further in our study  and through a detailed analysis  of the Nehari set associated to problem $(\mathcal P_\lambda)$ (see \cite{N1,N2}), we will show the existence of two critical hyperbolas on the plane $(a,b)$, that separates the plane into regions where the energy functional exhibits  distinct topological properties. Some of the ideas used here come from \cite{S,S1}, where the subcritical case was studied and a complete bifurcation diagram was provided. Our work contains  new results in the framework of Kirchhoff type equations with critical nonlinearity and extends the results of  \cite{DN} (for a detailed comparison see below). 

	To give a better description of our results, let us endow the Sobolev space $H^1_0(\Omega)$ with the
	classical norm $\|u\|=\left( \int_{\Omega }|\nabla u|^2 \ dx\right)^{\frac{1}{2}}$ and  denote by $\|u\|_{q}$ the Lebesgue norm in $L^{q}(\Omega)$ for $1\leq q \leq 2^*$,  i.e. $\|u\|_{q}=\left(\int_{\Omega} |u|^{q} \ dx\right)^{\frac{1}{q}}$.
	Let $S_N$ be the embedding constant of $H^1_0(\Omega)\hookrightarrow L^{2^*}(\Omega)$, i.e.
	\begin{equation}\label{embedding}S_N=\inf_{u\in H^1_0(\Omega) \setminus \{0\}}\frac{\|u\|^2}{ \|u\|^2_{2^*}} . \end{equation}
	Let us recall that
$$
	S_N=\frac{N(N-2)}{4}\omega_N^{\frac{2}{N}}
$$
	(where $\omega_N$ is the volume of the unit ball in $\R^N$) is sharp, but is never achieved unless $\Omega=\mathbb R^N$.
	
	For $N>4$ let us introduce the following constants whcih will have a crucial role in the sequel: 
$$
	C_1(N)=\frac{4(N-4)^\frac{N-4}{2}}{N^\frac{N-2}{2}S_N^\frac{N}{2}} \qquad \hbox{and} \qquad
	C_2(N)=\frac{2(N-4)^{\frac{N-4}{2}}}{(N-2)^{\frac{N-2}{2}}S_{N}^{\frac{N}{2}}},$$
and notice that $C_1(N)< C_2(N)$.

On the nonlinearity $f$ we will assume the following: 
\begin{enumerate}
	\item[(\mylabel{f1}{$\mathcal{F}_1$})] $f:\Omega\times \mathbb{R}\to \mathbb{R}$ is a Carath\'{e}odory function satisfying $f(x,0)=0$ for a.a. $x\in \Omega$;
	\item[(\mylabel{f2}{$\mathcal{F}_2$})] $f(x,v)> 0$ for every $v>0$ and a.a. $x\in \Omega$, $f(x,v)< 0$ for every $v<0$ and a.a. $x\in \Omega$. Moreover there exists $\mu>0$ such that  $f(x,v)\geq \mu >0$ for a.a. $x\in \Omega$ and every $v\in I$, being $I$ an open interval of $(0,+\infty)$;
	\item[(\mylabel{f3}{$\mathcal{F}_3$})]there exist $c>0$, $p\in (2,2^*)$ such that $|f(x,v)|\leq c(1+|v|^{p-1})$ for every $v\in \mathbb R$ and a.a. $x\in \Omega$;
	\item[(\mylabel{f4}{$\mathcal{F}_4$})] $f(x,v)=o(|v|)$ for $v\to 0$ and uniformly in $x\in \Omega$.
\end{enumerate}

Denote by $\Phi_\lambda:H^1_0(\Omega)\to\R$ the energy functional associated to $(\mathcal P_\lambda)$,
\[\Phi_\lambda(u)=\frac{a}{2}\|u\|^2+\frac{b}{4} \|u\|^4-\frac{1}{2^*}\|u\|^{2^*}_{2^*}-\lambda\int_\Omega F(x,u)dx \qquad \mbox{for every }\  u \in H^1_0(\Omega),\]
where
\begin{equation*}
F(x,v)=\int_0^vf(x,t)dt.
\end{equation*}
Note that from \eqref{f1} and \eqref{f3},  $\Phi_\lambda$ is well defined and $\Phi_\lambda\in C^1(H^1_0(\Omega))$.

\medskip

Our first result establishes the existence of global minimizers when $a^\frac{N-4}{2}b\geq C_1(N)$.
\begin{theor}\label{T1} Assume  \eqref{f1} - \eqref{f4}. 
	
	If $a^\frac{N-4}{2}b> C_1(N)$, then there exists $\lambda_0^*:=\lambda_0^*(a,b)>0$ such that:
	\begin{itemize}
		\item[i)]  For each $\lambda>\lambda_0^*$, problem $(\mathcal P_\lambda)$ has a non-zero solution $u_\lambda$, which is a global minimizer to $\Phi_\lambda$ with negative energy.
		\item[ii)]  Problem $(\mathcal P_{\lambda_0^*})$ has a non-zero solution $u_{\lambda_0^*}$, which is a global minimizer to $\Phi_{\lambda_0^*}$ with zero energy.
		\item[iii)] If $0<\lambda<\lambda_0^*$, then $\Phi_\lambda(u)>0$ for all $u\in H_0^1(\Omega)\setminus\{0\}$ and $0$ is a global minimizer of $\Phi_\lambda$.
	\end{itemize}
If $a^\frac{N-4}{2}b= C_1(N)$,   then for each $\lambda>0$, problem $(\mathcal P_\lambda)$ has a non-zero solution $u_\lambda$, which is a global minimizer to $\Phi_\lambda$ with negative energy. Furthermore, if $(a_k)_k,(b_k)_k$ are sequences satisfying $\displaystyle a_k^\frac{N-4}{2}b_k\downarrow C_1(N)$, $a_k\to a>0$ and $b_k\to b>0$, then $\lambda_0^*(a_k,b_k)\to 0$.
\end{theor}

 In the sequel, $\lambda_0^*$ is as in Theorem \ref{T1}. For $\lambda<\lambda_0^*$ but close to $\lambda_0^*$ we  can still prove the existence of a non trivial local minimizer provided that  $a^\frac{N-4}{2}b\ge C_2(N)$ as it is shown in the next result.
\begin{theor}\label{T12} Assume \eqref{f1} - \eqref{f4}. If $a^\frac{N-4}{2}b\ge C_2(N)$, then there exists $\varepsilon>0$ such that for each $\lambda_0^*-\varepsilon<\lambda<\lambda_0^*$,  problem $(\mathcal P_\lambda)$ has a non-zero solution $u_\lambda$, which is a local minimizer to $\Phi_\lambda$ with positive energy. Moreover $\Phi_\lambda(u)>0$ for all $u\in H_0^1(\Omega)\setminus\{0\}$.
\end{theor}
A second solution of $(P_\lambda)$ of  mountain pass type   is ensured by the next theorem. 
\begin{theor}\label{T2} Assume \eqref{f1} - \eqref{f4}. If $a^\frac{N-4}{2}b> C_2(N)$, then there exists $\varepsilon>0$ such that for each $\lambda>\lambda_0^*-\varepsilon$, problem $(\mathcal P_\lambda)$ has a non-zero solution $v_\lambda$, which is of a mountain pass type to $\Phi_\lambda$, with positive energy. If $a^\frac{N-4}{2}b= C_2(N)$, then the same result holds for $\lambda$ sufficiently large.
\end{theor}
For the next result, we need the additional hypothesis:
\begin{enumerate}
	\item[(\mylabel{f5}{$\mathcal{F}_5$})] For each $u\in H_0^1(\Omega)\setminus\{0\}$, the function $(0,\infty)\ni t\mapsto \int_\Omega f(x,tu(x)) dx$ is $C^1$.
\end{enumerate}
\begin{theor}\label{T3} Assume \eqref{f1} - \eqref{f5}. If $\displaystyle a^\frac{N-4}{2}b> C_2(N)$, then there exists $\lambda^*:=\lambda^*(a,b)\in (0,\lambda_0^*)$,  such that if $\lambda\in(0,\lambda^*)$, then $(\mathcal P_\lambda)$ has no non-zero solution. Moreover, there exists $u\in H_0^1(\Omega)\setminus\{0\}$ such that $\Phi'_\lambda(u)u=0$ if, and only if $\lambda\ge \lambda^*$.
\end{theor}
Now we focus on the power case $f(x,u)=|u|^{p-2}u$ with $p\in (2,2^*)$. In this case, some conclusions of Theorems \ref{T1}, \ref{T2} and \ref{T3} had already been established in \cite{DN}. Indeed, a comparison between the constants $\alpha_2$ (defined in \cite{DN}) and $C_2(N)$ shows (after some obvious modifications with respect to $a>0$) that $\alpha_2=C_2(N)$. Therefore \cite[Theorem B.8]{DN} corresponds to our Theorem \ref{T1} with the following observations:
\begin{itemize}
	\item[1)] In \cite[Theorem B.8]{DN} the existence of a global minimum of the energy functional $u_\lambda$ is only proved for $\alpha_2=C_2(N)\le a^\frac{N-4}{2}b$ and $\lambda$ sufficiently large in order to make the infimum negative, while in our case, we find $u_\lambda$ for all  $C_1(N)\le a^\frac{N-4}{2}b$ (remember that $C_1(N)<C_2(N)$) and there is a threshold $\lambda_0^*$ for the sign of the energy of $\Phi_\lambda$. Moreover, we proved the existence of a local minimizer with positive energy in case $\Phi_\lambda(u)>0$ for $u\neq 0$ and $a^\frac{N-4}{2}b\ge C_2(N)$ (see Theorem \ref{T12}).
	\item[2)] The arguments used in \cite[Theorem B.8]{DN}, to prove a mountain pass geometry to $\Phi_\lambda$ require $\lambda$ to be sufficiently large in order to make the infimum negative. We show that this geometry is preserved even in the case where $\Phi_\lambda(u)>0$ for all $u\in H_0^1(\Omega)\setminus\{0\}$ (see Theorems \ref{T12}, \ref{T2}).
	\item[3)] Theorem \ref{T3} was proved in \cite[Theorem B.8]{DN} for $\lambda$ sufficiently small. We also show that there exists $u\in H_0^1(\Omega)\setminus\{0\}$ such that $\Phi'_\lambda(u)u=0$ if, and only if $\lambda \ge \lambda^*$. However, when $f(x,u)=|u|^{p-2}u$ this result can be improved (see Theorem \ref{nonexistenceimproved}).
\end{itemize}

Concerning item 1), in fact, we have now a fairly complete result. Combining Theorem \ref{T1} with \cite[Proposition 4.2]{DN} we conclude that the curve $\displaystyle a^\frac{N-4}{2}b= C_1(N)$ is a threshold in the following sense:
\begin{theor}\label{T4} Suppose that $f(x,u)=|u|^{p-2}u$. If $0<\displaystyle a^\frac{N-4}{2}b\le C_1(N)$, then $\Phi_\lambda$ has a global minimizer with negative energy for all $\lambda>0$. If $\displaystyle a^\frac{N-4}{2}b> C_1(N)$, then $\Phi_\lambda$ has a global minimizer with negative energy if, and only if, $\lambda>\lambda_0^*(a,b)>0$, it has two global minimizers with zero energy for $\lambda=\lambda_0^*(a,b)$, and has zero as unique minimizer if $\lambda<\lambda_0^*(a,b)$. Moreover, if $(a_k)_k,(b_k)_k$ are sequences satisfying $\displaystyle a_k^\frac{N-4}{2}b_k\downarrow C_1(N)$, $a_k\to a>0$ and $b_k\to b>0$, then $\lambda_0^*(a_k,b_k)\to 0$. In all cases the global minimizer is a solution to problem $(\mathcal P_\lambda)$.
\end{theor}
Theorem \ref{T4} settles down the existence of global minimizers with negative energy for all ranges of  $a^\frac{N-4}{2}b$. It complements \cite[Theorem 1.2 and Theorem B.8]{DN}. In the power case it is also possible to improve Theorem \ref{T12} in the case $a^\frac{N-4}{2}b>C_1(N)$ and $\lambda<\lambda_0^*$ (in such case $\Phi_\lambda$ has zero as global minimizer):
\begin{theor}\label{T41}Suppose that $f(x,u)=|u|^{p-2}u$. If $C_1(N)<a^\frac{N-4}{2}b< C_2(N)$, then there exists $\varepsilon>0$ such that for each $\lambda_0^*-\varepsilon<\lambda<\lambda_0^*$,  problem $(\mathcal P_\lambda)$ has a non-zero solution $u_\lambda$, which is a local minimizer to $\Phi_\lambda$ with positive energy. Moreover $\Phi_\lambda(u)>0$ for all $u\in H_0^1(\Omega)\setminus\{0\}$.

\end{theor}

 Concerning the second solution, we complement \cite[Theorem 1.1]{DN} with the following results.
\begin{theor}\label{T5} Suppose that $f(x,u)=|u|^{p-2}u$ and $0<a^\frac{N-4}{2}b< C_2(N)$. Then there exists $p_0(a,b)\in(2,2^*)$ such that if $p\in(p_0(a,b),2^*)$, then for all $\lambda>0$, problem $(\mathcal P_\lambda)$ has a non-zero solution $v_\lambda$ with positive energy.
\end{theor}
\begin{theor}\label{T6} Suppose that $f(x,u)=|u|^{p-2}u$. For each $a,b>0$ there exists $\tilde\lambda:=\tilde\lambda(a,b,p)>0$ such that for all $\lambda>\tilde\lambda$, problem $(\mathcal P_\lambda)$ has a non-zero solution $v_\lambda$ with positive energy.
\end{theor}
We note here that in \cite[Theorem 1.1]{DN}, it was proved that for each fixed $p$, the conclusion of Theorem \ref{T5} holds true for sufficiently small $b$. We refer the reader to Theorem \ref{N-existence} and Remark \ref{comparisonnaimen}, in particular to item ii), where we show that the technique used to prove \cite[Theorem 1.1]{DN} (which we also used) can not hold for all values of $a,b,p$. However, the above theorem ensures that for each $p$ problem $(\mathcal P_\lambda)$ still has a second solution provided  $\lambda$ is big enough.

We conclude this work with an existence  result \`{a} la Brezis Nirenberg \cite{BN} which is a consequence of our study in the limit case ($b\downarrow 0$). 
\begin{theor}\label{T7} For each $\lambda>0$ and $p\in(2,2^*)$, the problem
	
	\begin{equation*}	(\mathcal Q_\lambda) \ \ \ \ \left\{
	\begin{array}{ll}
	- \Delta u=
	|u|^{2^*-2}u+\lambda |u|^{p-2}u, & \hbox{ in } \Omega,  \\
	u=0, & \hbox{on } \partial \Omega.
	\end{array}
	\right.
	\end{equation*}
	has  a nontrivial solution.
\end{theor}
The last remark of this Section explains  the reason why we focus on  positive parameters $\lambda$:
\begin{rem} If $\lambda\leq 0$, problem $(\mathcal P_\lambda)$ might have only the zero solution. Indeed, assume that $\Omega$ is a  star shaped domain and $f(v)=|v|^{p-2}v$ with $p\in (2,2^*)$. Then, if   $u$ is a solution of $(\mathcal P_\lambda)$ then $w=(a+b\|u\|^2)^{-\frac{1}{2^*-2}}u$ satisfies the equation  $-\Delta w=|w|^{2^*-2}w+\mu |w|^{p-2}w$ for some $\mu\leq 0$. Applying the Pohozaev identity we deduce that  $w=0$.
\end{rem}

The work is organized as follows:

\begin{itemize}
	\item in Section \ref{S2} we collect some prelimaries results that will be used throughout the work;
	\item in Section \ref{S3} we prove Theorems \ref{T1}, \ref{T12}, \ref{T2} and \ref{T3};
	\item in Section \ref{S4} we prove Theorems \ref{T4}, \ref{T41}, \ref{T5}, \ref{T6} and \ref{T7},
	\item in Appendix \ref{A} and \ref{AB} we present some technical results concerning the Nehari set associated to problem $(\mathcal{P}_\lambda)$ and $(\mathcal{P}_0)$ respectively.
\end{itemize}

\section{Preliminaries results}\label{S2}
In this Section we provide some auxiliary result which will be used throughout the work. Here only hypothesis \eqref{f1}-\eqref{f4} are used. For each $a,b>0$, define $g,h:(0,\infty)\to \mathbb{R}$ by
\begin{equation*}
g(t)=\frac{a}{2}+\frac{b}{4}t^2-S_N^{\frac{-2^*}{2}}\frac{t^{2^*-2}}{2^*},
\end{equation*}
\begin{equation*}
h(t)=a+bt^2-S_N^{\frac{-2^*}{2}}t^{2^*-2}.
\end{equation*}
A simple calculation shows that
\begin{lem}\label{functions} There holds:
	\begin{itemize}
		\item[i)]  $g$ has a unique local minimizer at
		\begin{equation}\label{t0}
		t_0=\left(\frac{2^* b }{2(2^*-2)}{S}_N^{\frac{2^*}{2}}\right)^{\frac{1}{2^*-4}}.
		\end{equation}
		Moreover, $g(t_0)>0$  if and only if  $a^\frac{N-4}{2}b> C_1(N)$, while if $a^\frac{N-4}{2}b= C_1(N)$, then  $g(t_0)=0$.
		\item[ii)] $h$ has a unique local minimizer at
		\begin{equation}\label{t00}
		t_0=\left(\frac{2 b }{2^*-2}{S}_N^{\frac{2^*}{2}}\right)^{\frac{1}{2^*-4}}.
		\end{equation}
		Moreover, $h(t_0)>0$ if and only if  $a^\frac{N-4}{2}b> C_2(N)$, while if  $a^\frac{N-4}{2}b= C_2(N)$, then
		 $h(t_0)=0$.
	\end{itemize}
\end{lem}
\begin{rem}\label{gh} Lemma \ref{functions} gives  the same conclusion if instead of $g,h$ we use $t^2g(t)$ and $t^2h(t)$. Indeed, note for example that $t^2g(t)=0$ and $(t^2g(t))'=0$ if, and only if, $g(t)=g'(t)=0$.
\end{rem}
As a consequence of Lemma \ref{functions} and Remark \ref{gh} we have
\begin{cor}\label{graph g} Suppose that $a^\frac{N-4}{2}b< C_2(N)$, then the function $\overline{g}(t)=t^2g(t)$ has only two critical points, $0 < t^-_{a,b} < t^+_{a,b} $. Moreover, $t^-_{a,b}$ is a local maximum and  $t^+_{a,b}$ is a local minimum with $\overline{g}''(t^-_{a,b})<0<\overline{g}''(t^+_{a,b})$. Furthermore if $a^\frac{N-4}{2}b= C_2(N)$, then the function $g(t)t^2$ is increasing and has a unique critical point at $t_{a,b}$ satisfying $\overline{g}''(t_{a,b})=0$ and
	\begin{equation*}
	g(t_{a,b})t_{a,b}^2=\frac{(2^*-2)^2a^2}{4 \cdot2^*(4-2^*)b}.
	\end{equation*}
\end{cor}
\begin{prop}\label{functionsappl} Suppose that $u\in H_0^1(\Omega)\setminus\{0\}$, then
	\begin{itemize}
		\item[i)] for all $t>0$ we have
			\begin{equation*}
		\frac{a}{2}\|u\|^2+\frac{b}{4}\|u\|^4t^2-\|u\|_{2^*}^{2^*}\frac{t^{2^*-2}}{2^*}>g(\|u\|t)\|u\|^2 ;
		\end{equation*}
			\item[ii)] for all $t>0$ we have \begin{equation*}
			a\|u\|^2+b\|u\|^4t^2-\|u\|_{2^*}^{2^*}t^{2^*-2}>h(\|u\|t)\|u\|^2.		
			\end{equation*}
			
	\end{itemize}
\end{prop}
\begin{proof} i) Indeed note that
	\begin{eqnarray*}
	t^2\left[\frac{a}{2}\|u\|^2+\frac{b}{4}\|u\|^4t^2-\|u\|_{2^*}^{2^*}\frac{t^{2^*-2}}{2^*}\right]&=&\frac{a}{2}(\|u\|t)^2+\frac{b}{4}(\|u\|t)^4-\frac{\|u\|_{2^*}^{2^*}}{\|u\|^{2^*}}\frac{(\|u\|t)^{2^*}}{2^*}\\ &>&\frac{a}{2}(\|u\|t)^2+\frac{b}{4}(\|u\|t)^4-{S}_N^{-\frac{2^*}{2}}\frac{(\|u\|t)^{2^*}}{2^*}, \qquad t>0.
	\end{eqnarray*}
	The conclusion follows from Lemma \ref{functions}. The strict inequality above is a consequence of the non existence of minimizers for \eqref{embedding}. The proof of ii) is similar.
\end{proof}

The next Lemma gives some important variational properties of the energy functional $\Phi_{\lambda}$.
\begin{lem}\label{variational property} The following  holds true. 
\begin{itemize}
	\item[1)] Let $a,b$ be positive numbers such that  $a^{\frac{N-4}{2}}b\geq C_{1}(N)$. Suppose that $\lambda_k\to \lambda\geq 0$ and $u_k\rightharpoonup u$. Then, $\Phi_{\lambda}(u)\leq \liminf_k\Phi_{\lambda_k}(u_k).$

 	\item[2)] Let $a,b$ be positive numbers such that  $a^{\frac{N-4}{2}}b\ge C_{2}(N)$. Suppose that $\lambda_k\to \lambda\geq 0$, $\Phi_{\lambda_k}(u_k)\to c\in \mathbb{R}$ and  $\Phi'_{\lambda_k}(u_k)\to 0$. If $a^{\frac{N-4}{2}}b= C_{2}(N)$ assume also that
 	\begin{equation*}
 	c\neq \frac{(2^*-2)^2a^2}{4 \cdot 2^*(4-2^*)b}.
 	\end{equation*}
 	 Then, $u_k$ has a convergent subsequence.
 	\item[3)] Let $a,b$ be positive numbers such that  $a^{\frac{N-4}{2}}b\geq C_{2}(N)$. Suppose that $\lambda_k\to \lambda\geq 0$ and $u_k\rightharpoonup u$. Then, $\Phi'_{\lambda}(u)(u)\leq \liminf_k\Phi'_{\lambda_k}(u_k)(u_k).$
 		\end{itemize}
 	\end{lem}
 	\begin{proof} Item 1) can be found, after some mild modifications, in \cite[Lemma 2.1]{FF}. In a similar way 3) can be proved. Item 2) follows easily from \cite[Lemma 2.2]{FF} when $a^{\frac{N-4}{2}}b> C_{2}(N)$ (see also \cite[Proposition B.1]{DN}). The case $a^{\frac{N-4}{2}}b= C_{2}(N)$  can be deduced from \cite[Proposition B.4]{DN}. Note from Corollary \ref{graph g} that
 		\begin{equation*}
 		\frac{(2^*-2)^2a^2}{4\cdot2^*(4-2^*)b}=g(t_{a,b})t_{a,b}^2,
 		\end{equation*}
 		and one can immediately see, after introducing the parameter $a$, that $g(t_{a,b})t_{a,b}^2=g(\tau_{b}^0)$, where $g(\tau_{b}^0)$ was defined in \cite[Lemma B3]{DN}.
 	\end{proof}
For each $\lambda\ge 0$ and $u\in H_0^1(\Omega)\setminus\{0\}$,  define the fiber maps associated to $\Phi_\lambda$,  $\psi_{\lambda,u}: \ (0,+\infty)\to\R$ by $$\psi_{\lambda,u}(t):=\Phi_\lambda(tu)= \frac{a}{2}\|u\|^2t^2+\frac{b}{4} \|u\|^4t^4-\frac{1}{2^*}\|u\|^{2^*}_{2^*}t^{2^*}-\lambda\int_\Omega F(x,tu)dx.$$
\begin{prop}\label{graph psi} Suppose $\lambda\geq 0$ and $u\in H_0^1(\Omega)\setminus\{0\}$, then
	\begin{itemize}
		\item[i)] there exists a neighborhood $V$ of the origin such that $\psi_{\lambda,u}(t)>0$ for all $t\in V\cap(0,+\infty)$. Moreover $\psi_{\lambda,u}(t)\to \infty$ as $t\to \infty$ and $\psi_{\lambda,u}$ is bounded from below;
		\item[ii)] there exists a neighborhood $V$ of the origin such that $\psi'_{\lambda,u}(t)>0$ for all $t\in V\cap(0,+\infty)$. Moreover $\psi'_{\lambda,u}(t)\to \infty$ as $t\to \infty$ and $\psi'_{\lambda,u}$ is bounded from below.
	\end{itemize}
\end{prop}
\begin{proof}
	
i) Note that	\begin{equation*}
	\psi_{\lambda,u}(t)=t^2\left(\frac{a}{2}\|u\|^2+\frac{b}{4} \|u\|^4t^2-\frac{1}{2^*}\|u\|^{2^*}_{2^*}t^{2^*-2}-\lambda\int_\Omega \frac{F(x,tu)}{t^2}dx\right).
	\end{equation*}
	From \eqref{f4} we conclude the existence of $V$. On the other hand we have
	\begin{equation*}
	\psi_{\lambda,u}(t)=t^4\left(\frac{a}{2}\|u\|^2t^{-2}+\frac{b}{4} \|u\|^4-\frac{1}{2^*}\|u\|^{2^*}_{2^*}t^{2^*-4}-\lambda\int_\Omega \frac{F(x,tu)}{t^4}dx\right).
	\end{equation*}
	Since $2<p<2^*<4$, we conclude from \eqref{f3} that $\psi_{\lambda,u}(t)\to \infty$ as $t\to \infty$. The last part is obvious.
	
	ii) Note that
	\begin{equation*}
	\psi'_{\lambda,u}(t)=t\left(a\|u\|^2+b\|u\|^4t^2-\|u\|^{2^*}_{2^*}t^{2^*-2}-\lambda\int_\Omega \frac{f(x,tu)u}{t}dx\right).
	\end{equation*}
	From \eqref{f4} again we conclude the existence of $V$. On the other hand we have
	\begin{equation*}
	\psi'_{\lambda,u}(t)=t^3\left(a\|u\|^2t^{-2}+b \|u\|^4-\|u\|^{2^*}_{2^*}t^{2^*-4}-\lambda\int_\Omega \frac{f(x,tu)}{t^3}dx\right).
	\end{equation*}
	Since $2<p<2^*<4$, we conclude from \eqref{f4} that $\psi'_{\lambda,u}(t)\to \infty$ as $t\to \infty$. The last part is obvious.
	\end{proof}
The remaining part of this Section is devoted to define a suitable extremal parameter $\lambda_0^*$ which will be crucial in our arguments.  Consider the system

\begin{equation}\label{zeroener}
\left\{	\begin{array}{ll}
\psi_{\lambda,u}(t)=0,\\
\psi'_{\lambda,u}(t)=0, \\
\psi_{\lambda,u}(t)=\inf_{s>0}\psi_{\lambda,u}(s).
\end{array} \right.
\end{equation}

\begin{prop}\label{lambdazero} Assume that $a^\frac{N-4}{2}b\ge C_1(N)$ and take $u\in H_0^1(\Omega)\setminus\{0\}$. Then  there exists a unique positive $\lambda_0(u)$  satisfying \eqref{zeroener}.

\end{prop}
\begin{proof} Note that
	\begin{equation}\label{fiberlambda}
	\psi_{\lambda,u}(t)- \psi_{\lambda',u}(t)=(\lambda'-\lambda)\int_\Omega F(x,tu)dx.
	\end{equation}
	Since $F(x,v)\geq 0$ for all $v\in \mathbb{R}$ (see \eqref{f2}), we conclude from \eqref{fiberlambda} that $\psi_{\lambda,u}(t)- \psi_{\lambda',u}(t)\ge 0$ for all $t\in \mathbb{R}$ and $0\le\lambda<\lambda'$. Moreover, on compact sets of the form $[c,d]$, with $0<c<d$, we deduce that $\psi_{\lambda,u}\to \psi_{\lambda',u}$ uniformly  as $\lambda\to \lambda'$. From Proposition \ref{graph psi}, there exists a neighborhood of the origin $V_{\lambda'}$ such that $\psi_{\lambda',u}(t)>0$ if $t\in V_{\lambda'}\cap(0,+\infty)$, therefore $\psi_{\lambda,u}(t)>0$ for all $0\le\lambda<\lambda'$. Once $\psi_{0,u}$ is positive on $(0,\infty)$ (see Proposition \ref{functionsappl}) and tends to $\infty$ as $t\to \infty$ we conclude that for $\lambda$ sufficiently small, the fiber map $\psi_{\lambda,u}$ is positive in $(0,\infty)$. On the other hand, fixed $t>0$ one can easily see that $\psi_{\lambda,u}(t)\to -\infty $ as $\lambda\to \infty$. Therefore, there exists a unique $\lambda_0(u)$ solving system \eqref{zeroener}.
	
	Now we claim that $\lambda_0(u)>0$. Indeed, from Lemma \ref{functions} and Proposition \ref{functionsappl} we have that
	\begin{equation*}
	\psi_{0,u}(t)> g(\|u\|t)(\|u\|t)^2\ge  0, \quad  \forall t>0.
	\end{equation*}
	From \eqref{fiberlambda} we conclude that $\lambda_0(u)>0$.

\end{proof}
\begin{rem}\label{lambdazeronegative} The proof of Proposition \ref{lambdazero} also shows that if $a^\frac{N-4}{2}b< C_1(N)$, then there exists $u\in H_0^1(\Omega)\setminus\{0\}$ such that $\lambda_0(u)<0$.
\end{rem}
\begin{prop}\label{prop systems} For each $u\in H^1_0(\Omega)\setminus \{0\}$ one has: $\lambda_0(u)$ is the unique parameter $\lambda>0$ for which the fiber map $\psi_{\lambda,u}$ has a critical point with zero energy and satisfies $\inf_{t>0}\psi_{\lambda,u}(t)=\inf_{t>0}\psi_{\lambda_0(u),u}(t)=0$. Moreover, if $\lambda>\lambda_0(u)$, then $\inf_{t>0} \psi_{\lambda,u}(t) < 0$ while if $0<\lambda\leq\lambda_0(u)$, then $\inf_{t>0} \psi_{\lambda,u}(t) = 0$.
	\end{prop}
\begin{proof}
 Choose any $t>0$ that solves \eqref{zeroener}. If $\lambda>\lambda_0(u)$, then $\psi_{\lambda,u}(t)<\psi_{\lambda_0(u),u}(t)=0$ and the claim follows. If $\lambda\leq\lambda_0(u)$, then $\psi_{\lambda,u}(t)\geq\psi_{\lambda_0(u),u}(t)\geq 0$ for all $t\ge 0$ and the conclusion follows at once.
	
\end{proof}
We introduce the following extremal parameter (see \cite{Y})
 \[\lambda_0^*=\inf_{u\in H^1_0(\Omega)\setminus \{0\}} \lambda_0(u).\]
\begin{prop}\label{comparison} The following holds true.
	\begin{itemize}
			\item[i)] If $\displaystyle a^\frac{N-4}{2}b>C_1(N)$, then $\lambda_0^*>0$.
		\item[ii)] If $\displaystyle a^\frac{N-4}{2}b=C_1(N)$,  then $ \lambda_0^*=0$. Moreover if $u_k\in H_0^1(\Omega)\setminus\{0\}$ satisfies $\lambda_0(u_k)\to \lambda_0^*=0$, then $u_k\rightharpoonup 0$ and $\frac{\|u_k\|_2^2}{\|u_k\|_{2^*}^2}\to S_N$.
	\end{itemize}
	
\end{prop}
\begin{proof} i) Let us prove that $\lambda_0^*>0$.  Notice first that the function $u\to\lambda_0(u)$ is zero homogeneous. Indeed, if $(t,\lambda_0(u))$ solves system \eqref{zeroener} and $\mu>0$, then
 \begin{equation*}
\left\{	\begin{array}{ll}
\psi_{\lambda,\mu u}(t)=\psi_{\lambda,u}(\mu t)=0,\\
\psi'_{\lambda,\mu u}(t)=\psi'_{\lambda, u}(\mu t)=0,
\end{array} \right.
\end{equation*}
by uniqueness, $\lambda(\mu u)=\lambda (u)$.	 We argue by contradiction assuming that $\lambda_0^*=0$. Then, there exists $\{u_k\}\subseteq H^1_0(\Omega)\setminus\{0\}$ such that $\lambda_k:=\lambda_0(u_k)\to 0.$ By homogeneity we can assume that $\|u_k\|=1$.
	 Then for each $k$, there exists $t_k>0$ such that
	 $\Phi_{\lambda_k}(t_ku_k)=\psi_{\lambda_k,u_k}(t_k)=0$ or equivalently
	\[\frac{a}{2}+\frac{b}{4}t_k^2-\frac{1}{2^*}\|u_k\|_{2^*}^{2^*} t_k^{2^*-2}-\lambda_k \int_\Omega \frac{F(x,t_ku_k)}{t_k^2}dx=0.\]
	 Thus, by Proposition \ref{functionsappl}, we obtain for each $k\in \mathbb N$
	\begin{equation}\label{limit}
	g(t_k)< \frac{a}{2}+\frac{b}{4}t_k^2-\frac{1}{2^*}\|u_k\|_{2^*}^{2^*} t_k^{2^*-2}\leq \lambda_k \int_\Omega \frac{F(x,t_ku_k)}{t_k^2}dx.
	\end{equation}
	Notice that from \eqref{f3} and \eqref{f4}, one has that for each $\varepsilon>0$ there exists $c>0$ such that $|f(x,v)|\leq \varepsilon |v|+c|v|^{p-1}$ for all $x\in\Omega$, $v\in \mathbb R$. Thus, $|F(x,v)|\leq \frac{\varepsilon}{2} v^2+\frac{c}{p}|v|^{p}$ for all $x\in\Omega$, $v\in \mathbb R$.  Hence, we deduce that $\{t_k\}$ is bounded in $(0,+\infty)$ and converge to some $\bar t>0$. Thus, from \eqref{limit} and Lemma \ref{functions} we deduce  that
	\[0<g(\bar t)\leq \lim_{k\to \infty}\lambda_k \int_\Omega \frac{F(x,t_ku_k)}{t_k^2}dx=0,\] which is a contradiction.
	
	
	ii) Without loss of generality we assume that $0\in \Omega$. Fix $\varphi\in C_0^\infty(\Omega)$ such that $\varphi\ge 0$ and $\varphi(x)=1$ in the open ball centered at $0$ of radius $R$ for some $R>0$. For each $\varepsilon>0$, define
	\begin{equation*}
	v_\varepsilon(x)=\frac{\varphi(x)}{(\varepsilon+|x|^2)^{\frac{N-2}{2}}}.
	\end{equation*}
	Let $u_\varepsilon=v_\varepsilon/\|v_\varepsilon\|$ and note that $u_{\varepsilon}\in H_0^1(\Omega)$ and (see \cite{BN})
	\begin{equation}\label{BNEQ}
	\|u_\varepsilon\|=1,\ \ \|u_\varepsilon\|_{2^*}^{2^*}=S_N^{\frac{-2^*}{2}}+O(\varepsilon^{\frac{2^*N}{4}}),\ \ \|v_\varepsilon\|=\frac{c}{\varepsilon^{\frac{N-2}{4}}}+k(\varepsilon),
	\end{equation}
	where $c>0$ does not depend on $\varepsilon$, $k(\varepsilon)>c_1>0$ for small $\varepsilon>0$, where $c_1$ is a constant and for every $q\in[2,2^*)$. Now given any $\lambda>0$ and fixed $t>0$, note that
	\begin{eqnarray*}
	\psi_{\lambda,u_\varepsilon}(t)&= &\frac{a}{2}t^2+\frac{b}{4}t^4-\frac{1}{2^*}\|u_\varepsilon\|_{2^*}^{2^*} t^{2^*}-\lambda \int_\Omega F(x,tu_\varepsilon)dx \\
	&=& t^2g(t)-\frac{1}{2^*}O(\varepsilon^{\frac{2^*N}{4}})t^{2^*}-\lambda \int_\Omega F(x,tu_\varepsilon)dx.
	\end{eqnarray*}
Take $t=t_0$ where $t_0$ is given by Lemma \ref{functions} and notice that, since $a^\frac{N-4}{2}b= C_1(N)$, then $g(t_0)=0$. 
We have that
\begin{equation*}
	\psi_{\lambda,u_\varepsilon}(t_0)=-\frac{1}{2^*}O(\varepsilon^{\frac{2^*N}{4}})t_0^{2^*}-\lambda \int_\Omega F(x,t_0u_\varepsilon)dx.
\end{equation*}
Let us estimate $\int_\Omega F(x,t_0u_\varepsilon)dx$ from below.
 By assumption \eqref{f2}, one has that $f(x,v)\geq \mu \chi_I(v)$ (being $\chi_I$ the characteristic function of the interval $I$), so there exist $\alpha, \beta >0$ such that $F(x,v)\geq \tilde F(v):=\mu\int_0^v \chi_I(t) dt \geq \beta $ for every $v\geq \alpha$.  Following Corollary 2.1 of \cite{BN} and using the positivity and  monotonicity of $F$, 
\begin{align*}
\int_{\Omega}F(x,t_0 u_\varepsilon)dx&\geq \int_{|x|\leq R}F(x,t_0 u_\varepsilon) dx \geq \int_{|x|\leq R}F\left(x, \frac{t_0}{\|v_\varepsilon\|(\varepsilon+|x|^2)^{\frac{N-2}{2}}}\right)dx \\&\geq \int_{|x|\leq R}\tilde F\left( \frac{t_0}{\|v_\varepsilon\|(\varepsilon+|x|^2)^{\frac{N-2}{2}}}\right)dx =c_1\varepsilon^\frac{N}{2}\int_0^{R\varepsilon^{-\frac{1}{2}} }\tilde F\left( \frac{t_0}{\|v_\varepsilon\|}\left(\frac{\varepsilon^{-1}}{1+s^2}\right)^{\frac{N-2}{2}}\right)s^{N-1}ds
\end{align*}
Notice that  
\begin{equation}\label{MM}
\tilde F\left(\frac{t_0}{\|v_\varepsilon\|}\left(\frac{\varepsilon^{-1}}{1+s^2}\right)^{\frac{N-2}{2}}\right)\geq \beta \hbox{ if $s$ is such that} \  \frac{t_0}{\|v_\varepsilon\|}\left(\frac{\varepsilon^{-1}}{1+s^2}\right)^{\frac{N-2}{2}} \geq\alpha.
\end{equation}
The second inequality of \eqref{MM} is equivalent to 
\begin{equation*}
\frac{t_0\varepsilon^{\frac{2-N}{4}}}{(c+\varepsilon^{\frac{N-2}{4}}k(\varepsilon))(1+s^2)^{\frac{N-2}{2}}}\ge \alpha,
\end{equation*}
 which is true if $s\le c_2\varepsilon^{-\frac{1}{4}}$ for some constant $c_2$ and small $\varepsilon$. Therefore, by taking a smaller $R$ if necessary, we deduce from \eqref{MM} that
$$\int_{\Omega}F(x,t_0 u_\varepsilon)dx\geq c_3 \varepsilon^\frac{N}{2}\int_0^{R\varepsilon^{-\frac{1}{4}} }\beta s^{N-1}ds=c_3 \varepsilon^{\frac{N}{4}},$$ for some positive constant $c_3$.
Thus, 
$$	\psi_{\lambda,u_\varepsilon}(t_0)\leq\varepsilon^{\frac{N}{4}}\left[-\frac{1}{2^*}\frac{O(\varepsilon^{\frac{2^*N}{4}})}{\varepsilon^{\frac{N}{4}}}t_0^{2^*}-\lambda c_3\right]<0,$$
 for small $\varepsilon$ and hence $\lambda_0(u_\varepsilon)<\lambda$. Once $\lambda$ was arbitrary we deduce that $\lambda_0^*=0$.

Now suppose that $u_k\in H_0^1(\Omega)\setminus\{0\}$ satisfies $\lambda_k:=\lambda_0(u_k)\to \lambda_0^*=0$. As in i) we may assume that $\|u_k\|=1$ and $u_k\rightharpoonup u$. Moreover there exists $t_k>0$ such that
\begin{equation*}
\frac{a}{2}+\frac{b}{4}t_k^2-\frac{1}{2^*}\|u_k\|_{2^*}^{2^*} t_k^{2^*-2}-\lambda_k\int_\Omega\frac{F(x,t_ku_k)}{t_k^2}dx=0 \qquad \hbox{for each} \ k\in \mathbb N.
\end{equation*}
From \eqref{f3} and \eqref{f4} we conclude that $t_k\to t>0$ and $\|u_k\|_{2^*}^{2^*}\to s>0$ and hence
\begin{equation*}
\frac{a}{2}+\frac{b}{4}t^2-\frac{1}{2^*}s t^{2^*-2}=0.
\end{equation*}
From the assumption on  $a$ and $b$ we conclude that $s=S_N^{\frac{-2^*}{2}}$ and hence $u_k$ is a minimizing sequence to $S_N$. Moreover, if $u\neq 0$, then (the first inequality is a consequence of Lemma \ref{functions} and the fact that  $\|u\|\le 1$)
\begin{eqnarray*}
0\le \frac{a}{2}+\frac{b}{4}t^2-\frac{S_N^{\frac{-2^*}{2}}}{2^*}\|u\|^{2^*} t^{2^*-2}&\le&  \frac{a}{2}+\frac{b}{4}t^2-\frac{1}{2^*}\|u\|_{2^*}^{2^*} t^{2^*-2} \\
	&\le & \liminf_{k\to \infty}\left(\frac{a}{2}+\frac{b}{4}t_k^2-\frac{1}{2^*}\|u_k\|_{2^*}^{2^*} t_k^{2^*-2}-\lambda_k\int_\Omega\frac{F(x,t_ku_k)}{t_k^2}dx\right) \\
	&=& 0,
\end{eqnarray*}
and consequently $u$ is a minimizer to $S_N$, which is an absurd, therefore $u=0$.

	\end{proof}
\begin{prop}\label{NehariP}  For each $\lambda\leq \lambda_0^*$ and each  $u\in H^1_0(\Omega)\setminus \{0\}$, $\inf_{t>0}\psi_{\lambda,u}(t)=0$; for each $\lambda>\lambda_0^*$ there exists $u\in H^1_0(\Omega)\setminus \{0\}$ such that $\Phi_\lambda(u)<0$.
\end{prop}
\begin{proof}
From Proposition \ref{prop systems},  if $\lambda\leq\lambda_0^*\leq \lambda_0(u)$, $\inf_{t>0}\psi_{\lambda,u}(t)=0$ for each  $u\in H^1_0(\Omega)\setminus \{0\}$; while if $\lambda>\lambda_0^*$, there exists $u\in H^1_0(\Omega)\setminus \{0\}$ such that $\inf_{t>0}\psi_{\lambda,u}(t)<0$ which implies at once the claim.
	
	\end{proof}
\section{Existence and non-existence results - General case}\label{S3}
In this Section we study the existence of global/local minimizers and mountain pass type solutions to $\Phi_\lambda$. At the end of the Section we show a non-existence result for small $\lambda>0$. We note here that in the first three subsections, only hypothesis \eqref{f1}-\eqref{f4} are needed, while in the fourth subsection we need to add hypothesis \eqref{f5}.
\subsection{Global minimizers for $\lambda\ge \lambda_0^*$} For each $\lambda>0$ define
\begin{equation*}
I_\lambda=\inf\{\Phi_\lambda(u):u\in H_0^1(\Omega)\}.
\end{equation*}

	\begin{theor}\label{existencem0} Suppose that $a^\frac{N-4}{2}b\ge C_1(N)$ and $\lambda>\lambda_0^*$. Then, there exists $u_\lambda\in H_0^1(\Omega)\setminus\{0\}$ such that $I_\lambda=\Phi_\lambda(u_\lambda)<0$.	
	\end{theor}
\begin{proof} In fact, one can easily see by using \eqref{f3}, \eqref{f4} and  the Sobolev embeddings  that $\Phi_\lambda$ is coercive. From Lemma \ref{variational property}  $\Phi_\lambda$ is also sequentially weakly lower semi-continuous and therefore by direct minimization arguments,  there exists $u_\lambda\in H_0^1(\Omega)$ such that $I_\lambda=\Phi_\lambda(u_\lambda)$. Moreover, from Proposition \ref{NehariP} there exists $w\in H_0^1(\Omega)$ such that $\Phi_\lambda(w)<0$, hence $I_\lambda<0$ and $u_\lambda\neq 0$.
\end{proof}
\begin{theor}\label{existencem1} Suppose that $a^\frac{N-4}{2}b\ge C_1(N)$ and $\lambda=\lambda_0^*$. The following holds true. 
	\begin{enumerate}
		\item[i)] If $a^\frac{N-4}{2}b> C_1(N)$, there exists $u_{\lambda_0^*}\in H_0^1(\Omega)\setminus\{0\}$ such that $I_{\lambda_0^*}=\Phi_{\lambda_0^*}(u_{\lambda_0^*})$. Moreover, $I_{\lambda_0^*}=0$.
		\item[ii)]  If $a^\frac{N-4}{2}b= C_1(N)$,  $u=0$ is the only minimizer for $I_{\lambda_0^*}$.
	\end{enumerate}	
\end{theor}	
\begin{proof} 	
i)	In fact, take a sequence $\lambda_k\downarrow \lambda_0^*$. From Theorem \ref{existencem0}, for each $k$, we can find $u_k\in H_0^1(\Omega)\setminus\{0\}$ such that $I_{\lambda_k}=\Phi_{\lambda_k}(u_k)<0$. Since $\lambda_k\downarrow \lambda_0^*$ it follows (as in the proof of Theorem \ref{existencem0}) that $\{u_k\}$ is bounded and therefore we may assume that $u_k\rightharpoonup u$ in $H_0^1(\Omega)$. From Lemma \ref{variational property} we obtain
	\begin{equation*}
	\Phi_{\lambda_0^*}(u)\le \liminf_{k\to \infty}\Phi_{\lambda_k}(u_k)\le 0.
	\end{equation*}
	Proposition \ref{NehariP} ensures that $\Phi_{\lambda_0^*}(w)\ge 0$ for each $w\in H_0^1(\Omega)$ and thus $\lim_{k\to \infty}\Phi_{\lambda_k}(u_k)=\Phi_{\lambda_0^*}(u)=0$, or  $I_{\lambda_0^*}=\Phi_{\lambda_0^*}(u)=0$.
	
	 To conclude the proof, we have to show that $u\neq 0$. In fact
 	\begin{equation*}
	\frac{a}{2}\|u_k\|^2+\frac{b}{4}\|u_k\|^4-\frac{S_N^{-\frac{2^*}{2}}}{2^*}\|u_k\|^{2^*}\\\le \frac{a}{2}\|u_k\|^2+\frac{b}{4}\|u_k\|^4-\frac{1}{2^*}\|u_k\|^{2^*}
	\le \lambda_k\int_\Omega{F(x,u_k)}dx.
	\end{equation*}
	Thus, \begin{equation*}
	g(\|u_k\|)=\frac{a}{2}+\frac{b}{4}\|u_k\|^2-\frac{S_N^{-\frac{2^*}{2}}}{2^*}\|u_k\|^{2^*-2}
	\le \lambda_k\int_\Omega\frac{F(x,u_k)}{\|u_k\|^2}dx.
	\end{equation*}
	If $u=0$, from \eqref{f3} and \eqref{f4}, the right hand side in the above inequality would tend to zero against the fact that $g(\|u_k\|)\geq \min_{[0,+\infty[}g>0$ (see Lemma \ref{functions}).

	ii) From Proposition \ref{comparison} we know that $\lambda_0^*=0$ and hence
	\begin{equation*}
	\Phi_{\lambda_0^*}(u)=\frac{a}{2}\|u\|^2+\frac{b}{4}\|u\|^4-\frac{1}{2^*}\|u\|_{2^*}^{2^*}.
	\end{equation*}
	The hypothesis $a^\frac{N-4}{2}b=C_1(N)$ implies that $u=0$ is the only minimizer for this functional. Indeed, from Proposition \ref{functionsappl} and Lemma \ref{functions} we have that
	\begin{equation*}
	\Phi_{\lambda_0^*}(u)>g(\|u\|)\|u\|^2\ge 0, \forall u\in H_0^1(\Omega)\setminus\{0\}.
	\end{equation*}
	
\end{proof}
\begin{prop}\label{lambdazeroachieved} Suppose that $a^\frac{N-4}{2}b> C_1(N)$. If $u\in H_0^1(\Omega)\setminus\{0\}$ satisfies $I_{\lambda_0^*}=\Phi_{\lambda_0^*}(u)$, then $\lambda_0^*=\lambda_0(u)$.
\end{prop}
\begin{proof} The equality $\lambda_0^*=\lambda_0(u)$ is a consequence of the definition of $\lambda_0^*$.
\end{proof}
\begin{theor}\label{continuitylambdazero}   If $\displaystyle a_k^\frac{N-4}{2}b_k\downarrow C_1(N)$, $a_k\to a>0$ and $b_k\to b>0$, then $\lambda_k:=\lambda_0^*\to 0$. Moreover, if $u_k\in H_0^1(\Omega)\setminus\{0\}$ satisfies $\lambda_k=\lambda_0(u_k)$, then $u_k\rightharpoonup 0$ and  $\frac{\|u_k\|_2^2}{\|u_k\|_{2^*}^2}\to S_N$.
	
\end{theor}
\begin{proof} For each $\varepsilon>0$, define $u_\varepsilon$ as in the proof of Proposition \ref{comparison}. Given any $\lambda>0$ and fixed $t>0$, note from \eqref{BNEQ} that
	\begin{eqnarray*}
		\psi_{\lambda,u_\varepsilon}(t)&= &\frac{a_k}{2}t^2+\frac{b_k}{4}t^4-\frac{1}{2^*}\|u_\varepsilon\|_{2^*}^{2^*} t^{2^*}-\lambda \int_\Omega F(x,tu_\varepsilon)dx \\
		&=& t^2g_k(t)-\frac{1}{2^*}O(\varepsilon^{\frac{2^*N}{4}})t^{2^*}-\lambda \int_\Omega F(x,tu_\varepsilon)dx,
	\end{eqnarray*} where $g_k$ is the analogous of $g$ with $a_k$ and $ b_k$ instead of $a$ and $b$.
	By taking $t=t_{0,k}$ where $t_{0,k}$ is given in \eqref{t0} (with $a_k$ and $ b_k$ instead of $a$ and $b$) we have that $t_{0,k}\to t_0>0$ ($t_0$ as in \eqref{t0}) and 	
	\begin{equation*}
\lim_k	\psi_{\lambda,u_\varepsilon}(t_{0,k})=
	\varepsilon^{\frac{N}{4}}\left[-\frac{1}{2^*}\frac{O(\varepsilon^{\frac{2^*N}{4}})}{\varepsilon^\frac{N}{4}}t_{0}^{2^*}-\lambda \int_\Omega \frac{F(x,t_{0}u_\varepsilon)}{\varepsilon^\frac{N}{4}}dx
\right],
	\end{equation*}
	 Since
	$$\int_{\Omega}F(x,t_{0} u_\varepsilon)dx\geq c \varepsilon^{\frac{N}{4}},$$ for some positive constant $c$,
	 we get  that $\psi_{\lambda,u_\varepsilon}(t_{0,k})<0$ for small $\varepsilon$ and big $k$. Then $\lambda_k\le\lambda_0(u_\varepsilon)<\lambda$. Once $\lambda$ was arbitrary we deduce that $\lambda_0^*=0$.
	
	Now suppose that $u_k\in H_0^1(\Omega)\setminus\{0\}$ satisfies $\lambda_k:=\lambda_0(u_k)\to \lambda_0^*=0$.  We may assume that $\|u_k\|=1$ and $u_k\rightharpoonup u$. Moreover there exists $t_k>0$ such that
	\begin{equation*}
	\frac{a_k}{2}+\frac{b_k}{4}t_k^2-\frac{1}{2^*}\|u_k\|_{2^*}^{2^*} t_k^{2^*-2}-\lambda_k\int_\Omega\frac{F(x,t_ku_k)}{t_k^2}=0.
	\end{equation*}
	From \eqref{f3} and \eqref{f4} we conclude that $t_k\to t>0$ and $\|u_k\|_{2^*}^{2^*}\to s>0$ and hence
	\begin{equation*}
	\frac{a}{2}+\frac{b}{4}t^2-\frac{1}{2^*}s t^{2^*-2}=0.
	\end{equation*}
	From the fact that  $a^\frac{N-4}{2}b=C_1(N)$ we infer that $s=S_N^{\frac{-2^*}{2}}$ and hence $(u_k)_k$ is a minimizing sequence to $S_N$. Moreover, if $u\neq 0$, then  (the first inequality is a consequence of Lemma \ref{functions} and the fact that  $\|u\|\le 1$)
	\begin{eqnarray*}
		0\le \frac{a}{2}+\frac{b}{4}t^2-\frac{S_N^{\frac{-2^*}{2}}}{2^*}\|u\|^{2^*} t^{2^*-2}&\le&  \frac{a}{2}+\frac{b}{4}t^2-\frac{1}{2^*}\|u\|_{2^*}^{2^*} t^{2^*-2} \\
		&\le & \liminf_{k\to \infty}\left(\frac{a_k}{2}+\frac{b_k}{4}t_k^2-\frac{1}{2^*}\|u_k\|_{2^*}^{2^*} t_k^{2^*-2}-\lambda_k\int_\Omega\frac{F(x,t_ku_k)}{t_k^2}dx\right) \\
		&=& 0,
	\end{eqnarray*}
	and consequently $u$ is a minimizer to $S_N$, which is an absurd, therefore $u=0$.
\end{proof}

\subsection{Mountain pass type solution for $\lambda\ge \lambda_0^*$}

\begin{prop}\label{MPG} For each $\lambda>0$, there exists $R_\lambda>0$ such that
	\begin{equation*}
	\inf\{\Phi_\lambda(u):\|u\|=R_\lambda\}>0.
	\end{equation*}
\end{prop}
\begin{proof} Indeed, given $\varepsilon>0$, from \eqref{f3}, \eqref{f4} and Sobolev embeddings, there exists a positive constant $c$ such that
	\begin{eqnarray*}
	\Phi_\lambda(u)&\ge& \frac{a}{2}\|u\|^2+\frac{b}{4} \|u\|^4-\frac{c}{2^*}\|u\|^{2^*}-\lambda c(\varepsilon\|u\|^2+\|u\|^p) \\
	&=& \left(\frac{a}{2}-\lambda c\varepsilon\right)\|u\|^2+\frac{b}{4} \|u\|^4-\frac{1}{2^*}\|u\|^{2^*}-\lambda c\|u\|^p, \forall u\in H_0^1(\Omega).
 	\end{eqnarray*}
 By choosing $\varepsilon>0$ conveniently the proof is complete.
\end{proof}
For each $\lambda\ge \lambda_0^*$ define 
\begin{equation*}
\Gamma_\lambda=\{\gamma\in C([0,1],H_0^1(\Omega)): \gamma(0)=0,\ \gamma(1)=u_{\lambda_0^*}\},
\end{equation*}
where $u_{\lambda_0^*}$ is as in Theorem \ref{existencem1}.
and
\begin{equation*}
c_\lambda=\inf_{\gamma\in \Gamma_\lambda}\max_{t\in [0,1]}\Phi_\lambda(\gamma(t)).
\end{equation*}
\begin{theor}\label{MPS} There holds:
	 \begin{itemize}\item[i)] If $a^\frac{N-4}{2}b> C_2(N)$, then for each $\lambda\ge \lambda_0^*$, there exist $w_\lambda\in H_0^1(\Omega)\setminus\{0\}$ such that $\Phi_\lambda(w_\lambda)=c_\lambda$ and $\Phi'_\lambda(w_\lambda)=0$.  \item[ii)]If $a^\frac{N-4}{2}b= C_2(N)$, then the above conclusion  holds for $\lambda$ sufficiently large.	
		\end{itemize}
\end{theor}
\begin{proof} The proof is standard and we write only the main steps. Note that $\Phi_\lambda(0)=0$ and $\Phi_\lambda(u_{\lambda_0^*})\le 0$. In fact, from Theorem \ref{existencem1} we know that $\Phi_{\lambda_0^*}(u_{\lambda_0^*})=0$ and if $\lambda>\lambda^*$, from Proposition \ref{prop systems} and Proposition \ref{lambdazeroachieved}, we must conclude that $\Phi_\lambda(u_{\lambda_0^*})<0$.
	 These together with Proposition \ref{MPG} implies a mountain pass geometry to $\Phi_\lambda$. 
	 
	 i) If $a^\frac{N-4}{2}b> C_2(N)$, from Lemma \ref{variational property},  $\Phi_\lambda$ satisfies the Palais-Smale condition at any level and the proof is complete.
	
ii) If	 $a^\frac{N-4}{2}b= C_2(N)$, it is enough to prove that (see Lemma \ref{variational property})
	\begin{equation*}
	c_\lambda\neq\frac{(2^*-2)^2a^2}{4 \cdot 2^*(4-2^*)b}.
	\end{equation*}
	We will actually show that  $c_\lambda\to 0$ as $\lambda\to \infty$. Indeed, given $\varepsilon>0$, fix any $\lambda'>0$. From \eqref{f1} and \eqref{f4}, there exists $\delta>0$ such that $0<\psi_{\lambda',u_{\lambda_0^*}}(t)\le \varepsilon$ for all $t\in(0,\delta]$. Since the function $(\lambda',\infty)\ni\lambda\mapsto \psi_{\lambda,u_{\lambda_0^*}}(\delta)$ is continuous, decreasing and tends to $-\infty$ as $\lambda\to \infty$ (see the proof of Proposition \ref{lambdazero}), it follows that there exists a unique parameter $\mu>\lambda'$ such that $\psi_{\mu,u_{\lambda_0^*}}(\delta)=0$. Now observe that on compact sets  $[t_0,t_1]\subset (0,\infty)$, we can always choose $\lambda$ so large that $\psi_{\lambda,u_{\lambda_0^*}}(t)<0$ for all $t\in[t_0,t_1]$. By taking $\delta$ even smaller if necessary, we can suppose that
	\begin{equation*}
c_\lambda \leq \max_{t\in[0,1]}\Phi_\lambda(tu_{\lambda_0^*})= \max_{t\in[0,1]}\psi_{\mu,u_{\lambda_0^*}}(t)=\max_{t\in(0,\delta)}\psi_{\mu,u_{\lambda_0^*}}(t)=\psi_{\mu,u_{\lambda_0^*}}(t_{max}),
	\end{equation*}
	where $t_{max}\in(0,\delta)$. Since  $\psi_{\mu,u_{\lambda_0^*}}(t_{max})\le \psi_{\lambda',u_{\lambda_0^*}}(t_{max})\le \varepsilon$, it follows that $c_\lambda\to 0$ as $\lambda\to \infty$. Choosing  $\lambda$ sufficiently large there holds
	\begin{equation*}
	c_\lambda<\frac{(2^*-2)^2a^2}{4\cdot 2^*(4-2^*)b},
	\end{equation*}
	and Lemma \ref{variational property} applies.
\end{proof}
\subsection{Local minimizers and mountain pass type solutions for $\lambda<\lambda_0^*$}
From Proposition \ref{NehariP},  $I_\lambda=\inf_{H^1_0(\Omega)}\Phi_\lambda\ge 0$ for $\lambda\le\lambda_0^*$, and consequently $u=0$ is a global minimizer of $\Phi_\lambda$. It is the unique global minimizer if $\lambda<\lambda_0^*$, while when $\lambda=\lambda_0^*$ (see Theorem \ref{existencem1}) there exists a second global minimizer $u_{\lambda_0^*}\neq 0$. We will prove that for $\lambda<\lambda_0^*$, close to $\lambda_0^*$, $\Phi_\lambda$ has a local minimizer  with positive energy.

First we prove a refined version of Proposition \ref{MPG}: fix $u_{\lambda_0^*}\in H_0^1(\Omega)\setminus\{0\}$ such that $\lambda_0^*=\lambda_0(u_{\lambda_0^*})$ (see Theorem \ref{existencem1} and Proposition \ref{lambdazeroachieved}). Denote $R=\|u_{\lambda_0^*}\|$.
\begin{prop}\label{MPGIMPROVED} Suppose that $\lambda\le \lambda_0^*$, then there exists $0<r<R$ and $M>0$ such that
	\begin{equation*}
	\inf\{\Phi_\lambda(u):u\in H_0^1(\Omega),\ \|u\|=r\}\ge M.
	\end{equation*}
\end{prop}
\begin{proof} Indeed, as in  the proof of Proposition \ref{MPG}, given $\varepsilon>0$, there exists a positive constant $c$, depending only on $N$ and $p$, such that
	\begin{equation*}
		\Phi_\lambda(u)\ge \left(\frac{a}{2}-\lambda c\varepsilon\right)\|u\|^2+\frac{b}{4} \|u\|^4-\frac{c}{2^*}\|u\|^{2^*}-\lambda c\|u\|^p, \ \forall u\in H_0^1(\Omega),
	\end{equation*}
	therefore
	\begin{equation*}
	\Phi_\lambda(u)\ge \left(\frac{a}{2}-\lambda_0^* c\varepsilon\right)\|u\|^2+\frac{b}{4} \|u\|^4-\frac{c}{2^*}\|u\|^{2^*}-\lambda_0^* c\|u\|^p, \ \forall u\in H_0^1(\Omega).
	\end{equation*}
	If we choose $\varepsilon$ in such a way that $\frac{a}{2}-\lambda_0^* c\varepsilon>0$ the proof is complete.
\end{proof}
Let $r$ be given as in Proposition \ref{MPGIMPROVED}. For each $\lambda\le\lambda_0^*$, define
\begin{equation*}
\hat{I}_\lambda=\inf\{\Phi_\lambda(u):u\in H_0^1(\Omega),\ \|u\|\ge r\}.
\end{equation*}
\begin{prop}\label{convergezero} Suppose that $a^\frac{N-4}{2}b> C_1(N)$, then $\hat{I}_\lambda\to 0$ as $\lambda\uparrow \lambda_0^*$.

\end{prop}
\begin{proof}  In fact, let $u\in H_0^1(\Omega)$ be such that $\lambda_0^*=\lambda_0(u)$ (see Theorem \ref{existencem1} and Proposition \ref{lambdazeroachieved}). Note that
	\begin{equation*}
	0\leq \hat{I}_\lambda \le \Phi_\lambda(u)\to 0,\ \mbox{as}\ \lambda\uparrow \lambda_0^*.
	\end{equation*}
	\end{proof}
\begin{theor}\label{existencelocal} Assume that $a^\frac{N-4}{2}b\ge C_2(N)$. There exists $\varepsilon>0$ such that if $\lambda\in (\lambda_0^*-\varepsilon,\lambda_0^*)$, then the infimum $\hat{I}_\lambda$ is achieved by some $u_\lambda\in H_0^1(\Omega)$ satisfying $\|u_\lambda\|>r$. Moreover $u_\lambda$ is a local minimizer and a critical point to $\Phi_\lambda$ and $\hat{I}_\lambda>0$.
\end{theor}
\begin{proof} From Proposition \ref{convergezero}, there exists $\varepsilon>0$ such that $\hat{I}_\lambda<\min\left\{M,\frac{(2^*-2)^2a^2}{42^*(4-2^*)b}\right\}$ for all $\lambda\in (\lambda_0^*-\varepsilon,\lambda_0^*)$, where $M$ is given by Proposition \ref{MPGIMPROVED}. Therefore, there exists $\delta>0$ such that if $(u_k)_k$ is a minimizing sequence to $\hat{I}_\lambda$, then $\|u_k\|>r+\delta$ for sufficiently large $k$. This combined with Ekeland's variational principle and Palais-Smale condition (see Lemma \ref{variational property}), implies the existence of $u_\lambda$ satisfying $\hat{I}_\lambda=\Phi_\lambda(u_\lambda)$ and $\|u_\lambda\|>r$. One can easily see that $u_\lambda$ is a local minimizer and a critical point to $\Phi_\lambda$. Moreover, from the definition of $\lambda_0^*$ we also have that $\hat{I}_\lambda>0$.

\end{proof}

Now we show the existence of a mountain pass type solution: let $\varepsilon>0$ be given as in Theorem \ref{existencelocal} and for each $\lambda\in (\lambda_0^*-\varepsilon,\lambda_0^*)$, choose $u_\lambda\in H_0^1(\Omega)\setminus\{0\}$ such that $\hat{I}_\lambda=\Phi_\lambda(u_\lambda)$. Define
\begin{equation*}
\Gamma_\lambda=\{\gamma\in C([0,1],H_0^1(\Omega)): \gamma(0)=0,\ \gamma(1)=u_\lambda\},
\end{equation*}
and
\begin{equation*}
c_\lambda=\inf_{\gamma\in \Gamma_\lambda}\max_{t\in [0,1]}\Phi_\lambda(\gamma(t)).
\end{equation*}
\begin{theor}\label{MPS>0} Assume that $a^\frac{N-4}{2}b> C_2(N)$, then for each $\lambda\in (\lambda_0^*-\varepsilon,\lambda_0^*)$, there exist $w_\lambda\in H_0^1(\Omega)\setminus\{0\}$ such that $\Phi_\lambda(w_\lambda)=c_\lambda$ and $\Phi'_\lambda(w_\lambda)=0$.
\end{theor}
\begin{proof} Note that $\min\{\Phi_\lambda(0),\Phi_\lambda(u_\lambda)\}< M$, where $M$ is given by Proposition \ref{MPG}. Therefore $\Phi_\lambda$ has a mountain pass geometry. From Lemma \ref{variational property} we know that $\Phi_\lambda$ satisfies the Palais-Smale condition and thus the proof is complete.
\end{proof}
Now we are in position to prove Theorems \ref{T1}, \ref{T12}, \ref{T2}:
\begin{proof}[Proof of Theorem \ref{T1}] It follows from Theorems \ref{existencem0}, \ref{existencem1}, \ref{continuitylambdazero} and  the definition of $\lambda_0^*$. 
\end{proof}
\begin{proof}[Proof of Theorem \ref{T12}] It follows from Theorem \ref{existencelocal}. \end{proof}
\begin{proof}[Proof of Theorem \ref{T2}] It follows from Theorems \ref{MPS} and \ref{MPS>0}.
\end{proof}
\subsection{Non-existence result} Suppose \eqref{f5}. Therefore the following system is well defined:
\begin{equation}\label{zeroenergyderivati}
\left\{	\begin{array}{ll}
\psi'_{\lambda,u}(t)=0,\\
	\psi''_{\lambda,u}(t)=0, \\
	\psi'_{\lambda,u}(t)=\inf_{s>0}\psi'_{\lambda,u}(s).
		\end{array} \right.
\end{equation}
The next Proposition can be proved in the same way as Proposition \ref{lambdazero}
\begin{prop} Assume that $u\in H_0^1(\Omega)\setminus\{0\}$, then  there exists a unique $\lambda(u)>0$  satisfying \eqref{zeroenergyderivati}.
	
\end{prop}
\begin{prop}\label{prop systems 1} For each $u\in H^1_0(\Omega)\setminus \{0\}$ one has: $\lambda(u)$ is the unique parameter $\lambda>0$ for which the fiber map $\psi_{\lambda,u}$ has a critical point with second derivative zero and satisfies $\inf_{t>0}\psi'_{\lambda,u}(t)=0$. Moreover, if $0<\lambda<\lambda(u)$, then $\psi_{\lambda,u}$ has no critical points.
\end{prop}
\begin{proof}  If $0<\lambda<\lambda(u)$, then  $\psi'_{\lambda,u}(s)>\psi'_{\lambda(u),u}(s)\geq 0$ for each $t>0$.
\end{proof}

\begin{cor}\label{lambda<lambdazero} For each $u\in H^1_0(\Omega)\setminus \{0\}$ one has that $\lambda(u)<\lambda_0(u)$.
\end{cor}
\begin{proof} Indeed, assume on the contrary that $\lambda_0(u)\le \lambda(u)$, then from Proposition \ref{prop systems 1}, the definition of $\lambda_0(u)$ and Proposition \ref{graph psi}, we deduce that $\psi_{\lambda_0(u),u}$ is increasing, which contradicts the definition of $\lambda_0(u)$, therefore, $\lambda(u)<\lambda_0(u)$.
\end{proof}
Define the extremal value (see \cite{Y})
\[\lambda^*=\inf_{u\in H^1_0(\Omega)\setminus \{0\}} \lambda(u).\]
\begin{prop}\label{comparison1} There holds:
	\begin{itemize}
			\item[i)] If $\displaystyle a^\frac{N-4}{2}b> C_2(N)$, then $0<\lambda^*< \lambda_0^*$.
						\item[ii] If $\displaystyle a^\frac{N-4}{2}b=C_2(N)$, then $ \lambda^*=0$. Moreover if $u_k\in H_0^1(\Omega)\setminus\{0\}$ satisfies $\lambda(u_k)\to \lambda^*=0$, then $u_k\rightharpoonup 0$ and $\frac{\|u_k\|_2^2}{\|u_k\|_{2^*}^2}\to S_N$.
	\end{itemize}
	
\end{prop}
\begin{proof} We only prove that $\lambda^*< \lambda_0^*$ (the rest of the proof is similar to the proof of Proposition \ref{comparison}). Indeed, from Theorem \ref{existencem1} and Proposition \ref{lambdazeroachieved}, there exists $u\in H_0^1(\Omega)\setminus\{0\}$ such that $\lambda_0^*=\lambda_0(u)$, therefore from Corollary \ref{lambda<lambdazero} we obtain $\lambda^*\le \lambda(u)<\lambda_0(u)=\lambda^*_0$.

\end{proof}
\begin{prop}\label{NehariP1} For each $\lambda<\lambda^*$, the fiber map $\psi_{\lambda,u}$ is increasing and has no critical points.
\end{prop}
\begin{proof} This follows form the fact that $\lambda<\lambda^*\leq \lambda(u)$ for every $u\in H_0^1(\Omega)\setminus\{0\}$ and Proposition \ref{prop systems 1}.
\end{proof}
	\begin{theor}\label{nonexistence} If $\displaystyle a^\frac{N-4}{2}b> C_2(N)$ and $\lambda\in(0,\lambda^*)$, then $(\mathcal P_\lambda)$ has no  non-zero solution.
		\end{theor}
	\begin{proof} In fact, from Proposition \ref{NehariP1} we have that $\psi'_{\lambda,u}(t)>0$ for all $t>0$ and $u\in H_0^1(\Omega)\setminus\{0\}$, therefore $\Phi_\lambda$ has no critical points other than $u=0$.
	\end{proof}
The next result provides the existence of $u\in H_0^1(\Omega)\setminus\{0\}$ such that $\Phi_{\lambda^*}'(u)u=0$.

	\begin{prop}\label{lambdaachieved} Suppose that $a^{\frac{N-4}{2}}b>C_2(N)$. Then, there exists $u\in H^1_0(\Omega)\setminus \{0\}$ such that $\lambda^*=\lambda(u)$.
\end{prop}
\begin{proof}
	Let $\lambda_k$ be a sequence of positive numbers converging to $ \lambda^*$. Thus, there exists  $u_k\in H^1_0(\Omega)\setminus \{0\}$ with $\|u_k\|=1$ (by the  homogeneity of the map $u\to\lambda(u)$) such that $\lambda_k=\lambda(u_k)$. We deduce then, the existence of  $u\in H^1_0(\Omega)$ such that $u_k\rightharpoonup u$. We claim that $u\neq 0$.
	By the defintion of $\lambda_k$, there exists $t_k=t(u_k)>0$ such that
	\[\psi'_{\lambda_k, u_k}(t_k)=\Phi'_{\lambda_k}(t_ku_k)( u_k)=0\] that is
	\[	a+bt_k^2-\|u_k\|_{2^*}^{2^*} t_k^{2^*-2}-\lambda_k \int_\Omega \frac{f(x,t_ku_k)u_k}{t_k}dx=0.\]
	
	Thus, we obtain
	\begin{equation}\label{null derivative2} 0<h(t_k)\le  a+bt_k^2-S_N^{-\frac{2^*}{2}} t_k^{2^*-2}\leq \lambda_k\int_\Omega \frac{f(x,t_ku_k)u_k}{t_k}dx.	\end{equation}
 From the above inequality, \eqref{f3} and \eqref{f4} we deduce that $\{t_k\}$ is bounded in $(0,+\infty)$ and it admits a subsequence still denoted by $\{t_k\}$ converging to some $\bar t>0$. Also, from  \eqref{null derivative2} and Lemma \ref{functions} we deduce that $u\neq0$. By Proposition \ref{prop systems 1}, $\psi'_{\lambda^*, u}(t)>0$ for every $t>0$. But since $t_ku_k\rightharpoonup \bar t u$, by 3) Lemma \ref{variational property} it follows
	$$\psi'_{\lambda^*, u}(\bar t)=\Phi_{\lambda^*}'(\bar t u)(\bar t u)\leq \liminf_k\Phi_{\lambda_k}'(t_k u_k)(t_k u_k)=\liminf_k \psi'_{\lambda_k, u_k}(t_k)=0,$$ which leads to a contradiction.
\end{proof}
As a consequence we have:
\begin{proof}[Proof of Theorem \ref{T3}]  It follows from Theorem \ref{nonexistence}, Proposition \ref{prop systems 1} and Proposition \ref{lambdaachieved}.
\end{proof}
\section{A particular case: $f(x,u)=|u|^{p-2}u$}\label{S4}
In this Section we consider the particular case where $f(x,u)=|u|^{p-2}u$, that is

\begin{equation}
\label{parti} \ \ \ \ \left\{
\begin{array}{ll}
- \left( a+b\ds\int_\Omega |\nabla u|^2 dx\right)\Delta u=
|u|^{2^*-2}u+\lambda |u|^{p-2}u, & \hbox{ in } \Omega \\ \\
u=0, & \hbox{on } \partial \Omega
\end{array}
\right.
\end{equation}
and $p\in (2,2^*)$. We will compare the results obtained here with  the literature. In fact we will extend and complement some results of \cite{DN}. For some values of $p$  in fact, we have a fairly complete picture. One can easily see that $f(x,u)=|u|^{p-2}u$ satisfies all hypothesis \eqref{f1}-\eqref{f5} and therefore, with respect to problem \eqref{parti} we have, as a consequence of Theorems \ref{T1}, \ref{T12}, \ref{T2} and \ref{T3}, the following:
\begin{theor}\label{MT} There exists a function $\lambda_0^*:(0,\infty)^2\to [0,\infty)$ satisfying the following. 
	\begin{itemize}
		\item[i)] If $\displaystyle a^\frac{N-4}{2}b> C_1(N)$, then $\lambda_0^*(a,b)>0$ and:
		\begin{itemize}
			\item[1)]  For each $\lambda> \lambda_0^*(a,b)$, problem \eqref{parti} admits a positive solution, which is a global minimizer to $\Phi_\lambda$ with negative energy.
			\item[2)]  If $\lambda=\lambda_0^*(a,b)$, then problem \eqref{parti} admits a positive solution, which is a global minimizer to $\Phi_{\lambda_0^*(a,b)}$ with zero energy.
			\item[3)] For $\lambda\in(0,\lambda_0^*(a,b))$, then only global minimizer to $\Phi_\lambda$ is $u=0$.
		\end{itemize}
		\item[ii)]  If $\displaystyle a^\frac{N-4}{2}b= C_1(N)$, then $\lambda_0^*(a,b)=0$ and for each $\lambda>0$, problem \eqref{parti} admits a positive solution, which is a global minimizer to $\Phi_\lambda$ with negative energy.
		\item[iii)] Moreover
	\begin{equation*}
	\lambda_0^*(a_k,b_k)\to 0,\ \mbox{if}\ a_k\to a>0, b_k\to b>0,\  \displaystyle a_k^\frac{N-4}{2}b_k\downarrow C_1(N).
	\end{equation*}
	\item[iv)] If $\displaystyle a^\frac{N-4}{2}b\ge C_2(N)$, then there exists $\varepsilon:=\varepsilon(a,b)>0$ such that: for each $\lambda \in(\lambda_0^*(a,b)-\varepsilon,\lambda_0^*(a,b))$, problem \eqref{parti} admits a positive solution, which is a local minimizer to $\Phi_\lambda$ with positive energy.
	\end{itemize}
\end{theor}
Recall that $C_1(N)<C_2(N)$.
\begin{theor}\label{MT1} There exists a function $\lambda^*:(0,\infty)^2\to [0,\infty)$ satisfying the following. 
	\begin{itemize}
		\item[i)] If $\displaystyle a^\frac{N-4}{2}b> C_2(N)$, then $0<\lambda^*(a,b)<\lambda_0^*(a,b)$.
		\item[ii)] If $\displaystyle a^\frac{N-4}{2}b= C_2(N)$, then $0=\lambda^*(a,b)<\lambda_0^*(a,b)$.
		\item[iii)] If $\displaystyle a^\frac{N-4}{2}b> C_2(N)$, then there exists $\varepsilon:=\varepsilon(a,b)>0$ such that for each $\lambda>\lambda_0^*(a,b)-\varepsilon$, problem \eqref{parti} admits a positive mountain pass type solution with positive energy.
		\item[iv) ]If $\displaystyle a^\frac{N-4}{2}b= C_2(N)$, then there exists $\tilde\lambda>0$ such that for each $\lambda>\tilde\lambda$, problem \eqref{parti} admits a positive mountain pass type solution with positive energy.
		\item[v)] If $\displaystyle a^\frac{N-4}{2}b> C_2(N)$ and $\lambda\in(0,\lambda^*(a,b))$, then problem \eqref{parti} has no non-zero solutions.
	
	\end{itemize}
\end{theor}
We note that items i) and ii) of Theorem \ref{MT1} follow from Proposition \ref{comparison1}. Combining Theorem \ref{MT} with \cite[Proposition 4.2]{DN} we conclude that the curve $\displaystyle a^\frac{N-4}{2}b= C_1(N)$ is a threshold in the sense stated in Theorem \ref{T4}:

\begin{proof}[Proof of Theorem \ref{T4}] By inspection, one can easily see that the constant $\alpha_2$ defined in \cite{DN} corresponds to our $C_2(N)$ with obvious modifications with respect to $a> 0$. Since $C_1(N)<C_2(N)$ and for each $a,b$ satisfying $0<\displaystyle a^\frac{N-4}{2}b\le C_1(N)$, there exists $u\in H_0^1(\Omega)\setminus\{0\}$ such that $\Phi_\lambda(u)<0$ for all $\lambda>0$, 
	it follows that  \cite[Proposition 4.2]{DN} can be applied and then $\Phi_\lambda$ has a global minimizer with negative energy for all $\lambda>0$. The rest of the proof is a consequence of Theorem \ref{MT}.
\end{proof}

In order to get more results concerning our problem $(\mathcal{P}_\lambda)$, let us introduce and study the Nehari sets associated to $\Phi_\lambda$: for each $a,b,\lambda\in \mathbb{R}$ let
$$\mathcal N:=\mathcal N_{a,b,\lambda} = \{ u \in H^1_0(\Omega)\setminus \{ 0 \} : \Phi'_\lambda(u)u=0\}=\{ u \in H^1_0(\Omega)\setminus \{ 0 \} : \psi'_{\lambda,u}(1)=0\}.$$	
We split the above set in three disjoint sets
$$\mathcal N^0:=\mathcal N_{a,b,\lambda}^0=\{ u \in H^1_0(\Omega)\setminus \{ 0 \} : \psi'_{\lambda,u}(1)=0, \psi''_{\lambda,u}(1)=0\},$$
$$\mathcal N^+:=\mathcal N_{a,b,\lambda}^+=\{ u \in H^1_0(\Omega)\setminus \{ 0 \} : \psi'_{\lambda,u}(1)=0, \psi''_{\lambda,u}(1)>0\},$$
$$\mathcal N^-:=\mathcal N_{a,b,\lambda}^-=\{ u \in H^1_0(\Omega)\setminus \{ 0 \} : \psi'_{\lambda,u}(1)=0, \psi''_{\lambda,u}(1)<0\}.$$
By using the implicit function theorem and the Lagrange's multiplier rule we have that:
\begin{prop}\label{naturalconstraint} Suppose that $a,b>0$ and $\lambda\ge 0$. Then, whenever  $\mathcal{N}^-,\mathcal{N}^+$ are not empty, they are $C^1$ manifolds of co-dimension $1$ in $H_0^1(\Omega)$. Moreover, every critical point of $\Phi_\lambda$ restricted to $\mathcal{N}^-\cup \mathcal{N}^+$ is a critical point to $\Phi_\lambda$. Moreover, if $u\in \mathcal N^+$ is a local minimizer of ${\Phi_\lambda}_{|_{\mathcal N^+}}$, then it is a local minimizer of $\Phi_\lambda$ over $H^1_0(\Omega)$.
\end{prop}
To understand the Nehari sets we prove:
\begin{prop}\label{graph psi 1}
	For each $a,b>0$ and $\lambda\ge0$ and $u\in H^1_0(\Omega),$ only one of the next $i)-iii)$ occurs.
	
	\begin{itemize}
		\item[i)] The function $\psi_{\lambda,u}$ is increasing and has no critical points.
		\item[ii)] The function $\psi_{\lambda,u}$ has only one critical point in $(0,+\infty)$ at the value $t_\lambda(u)$. Moreover, $\psi''_{\lambda,u}(t_\lambda(u))=0$ and $\psi_{\lambda,u}$ is increasing.
		\item[iii)] The function $\psi_{\lambda,u}$ has only two critical points, $0 < t^-_\lambda (u) < t^+_\lambda (u)$. Moreover, $t^-_\lambda(u)$ is a local maximum and  $t^+_\lambda(u)$ is a local minimum with $\psi_{\lambda,u}''(t^-_\lambda(u))<0<\psi_{\lambda,u}''(t^+_\lambda(u))$.
	\end{itemize}
\end{prop}
\begin{proof}
	We have $\psi_{\lambda,u}'(t)=0$ if and only if
	$$a\|u\|^2=-b\|u\|^4t^2+\|u\|^{2^*}_{2^*}t^{2^*-2}+\frac{\lambda}{p}\|u\|_p^pt^{p-2}.$$
	Let $\varphi(t)=-b\|u\|^4t^2+\|u\|^{2^*}_{2^*}t^{2^*-2}+\frac{\lambda}{p}\|u\|_p^pt^{p-2}$ for each $t>0$. Then, it is easy to see that there exists a unique maximum point $t^*$ of $\varphi$ such that $\varphi(t^*)>0$. Thus, the following cases occur.
	If $a\|u\|^2>\varphi(t^*)$, then,  	$\psi_{\lambda,u}'(t)>0$ for every $t>0$ and i) holds. If $a\|u\|^2=\varphi(t^*)$, then, $\psi_{\lambda,u}'(t)>0$ for every $t\neq t^*$ and $\psi_{\lambda,u}''(t^*)=a\|u\|^2-\varphi(t^*)-t^* \varphi'(t^*)=0$, so that ii) is verified. Finally, if $a\|u\|^2<\varphi(t^*)$, then, there exist $t_1<t^*<t_2$ such that $a\|u\|^2=\varphi(t_1)=\varphi(t_2)$ and $a\|u\|^2>\varphi(t)$ for $t<t_1$ and $t>t_2$, $a\|u\|^2<\varphi(t)$ for $t_1<t<t_2$ so that iii) is satisfied with $t_\lambda^-(u)=t_1$ and $t_\lambda^+(u)=t_2$.
\end{proof}
\subsection{A refined non-existence result}
Recall from Theorem \ref{nonexistence} that if $\displaystyle a^\frac{N-4}{2}b> C_2(N)$ and $\lambda\in(0,\lambda^*)$, then $(\mathcal P_\lambda)$ has no  non-zero solution. This is clear, since for that range of parameters, the Nehari set is empty. We show how to improve the non-existence result. First we need some preliminaries results:
\begin{cor}\label{equation} Assume that $\displaystyle a^\frac{N-4}{2}b> C_2(N)$, then for each $u\in H_0^1(\Omega)\setminus\{0\}$ satisfying $\lambda^*=\lambda(u)$ we have that
	\begin{equation*}
	-(2a+4b\|u\|^2)\Delta u-2^*|u|^{2^*-2}u-\lambda^* p|u|^{p-2}u=0.
	\end{equation*}
\end{cor}
\begin{proof} Define $J_{\lambda^*}:H^1_0(\Omega)\to \mathbb R$ by $J_{\lambda^*}(w)=\Phi_{\lambda^*}'(w)w$. From Lemma \ref{variational property} item 3),  $J_{\lambda^*}$ attains its infimum. Moreover, by the definition of $\lambda^*$, 
	\begin{equation*}
	\inf\{J_{\lambda^*}(w):w\in H_0^1(\Omega)\}=J_{\lambda^*}(u).
	\end{equation*} (see also Proposition \ref{lambdazeroachieved}).
 We conclude that  $J_{\lambda^*}'(u)=0$, which is the desired equation.
\end{proof}

\begin{theor}\label{nonexistenceimproved} If $a^{\frac{N-4}{2}}b>C_{2}(N)$ and $\Omega$ is star-shaped, then there exists $\varepsilon>0$ such that $(\mathcal{P}_\lambda)$ has no non-zero solution for each $\lambda\in(0,\lambda^*+\varepsilon)$.
\end{theor}
\begin{proof} The case  $\lambda\in(0,\lambda^*)$ is given by Theorem \ref{nonexistence}. Suppose on the contrary that $(\mathcal{P}_{\lambda^*})$ has a non-zero solution $u$. From Proposition \ref{prop systems 1} and the definition of $\lambda^*$, we have that $u\in \mathcal{N}_{\lambda^*}^0=\mathcal{N}_{\lambda^*}$ (note from Proposition \ref{lambdaachieved} that $\mathcal{N}_{\lambda^*}^0\neq \emptyset$) and hence $\lambda^*=\lambda(u)$. From Corollary \ref{equation} we deduce that
	\begin{equation*}
	\left\{	\begin{array}{ll}
	-(a+b\|u\|^2)\Delta u-|u|^{2^*-2}u-\lambda^* |u|^{p-2}u=0,\\ \\
	-(2a+4b\|u\|^2)\Delta u-2^*|u|^{2^*-2}u-\lambda^* p|u|^{p-2}u=0,
	\end{array} \right.
	\end{equation*}	
	which implies that
	\begin{equation*}
	-[(2-p)a+(4-p)b\|u\|^2]\Delta u=(2^*-p)|u|^{2^*-2}u,
	\end{equation*}
which leads, from   	Pohozaev identity, to $u=0$, a contradiction. Now suppose that there exists a sequence $\lambda_k\downarrow \lambda^*$ and a corresponding sequence of non-zero solutions $u_k$ of $(\mathcal{P}_{\lambda_k})$. Then
	\begin{equation*}
		a+b\|u_k\|^2 -\|v_k\|_{2^*}^{2^*}\|u_k\|^{2^*-2} -\lambda_k \|v_k\|_p^p\|u_k\|^{p-2} =0,
	\end{equation*}
	where $v_k=u_k/\|u_k\|$. Therefore $(u_k)_k$ is bounded and does not converge to $0$. From Lemma \ref{variational property} item 2), we conclude that $u_k\to u\in H_0^1(\Omega)\setminus\{0\}$ and
	\begin{equation*}
	-(a+b\|u\|^2)\Delta u-|u|^{2^*-2}u-\lambda^* |u|^{p-2}u=0,
	\end{equation*}
	that is $u$ is a non zero solution of $(\mathcal{P}_{\lambda^*})$, a contradiction.
	
\end{proof}
\subsection{Existence of local minimizers with positive energy when $C_1(N)<a^\frac{N-4}{2}b< C_2(N)$}
In this Section we prove Theorem \ref{T41}. For each $a,b,\lambda> 0$, define (whenever $\mathcal{N}^0_\lambda$, $\mathcal{N}^+_\lambda$ are not empty)
\begin{equation*}
c^0:=c^0(a,b,\lambda)=\inf\{\Phi_\lambda(u):u\in \mathcal N^0\},
\end{equation*}
\begin{equation*}
c^+:=c^+(a,b,\lambda)=\inf\{\Phi_\lambda(u):u\in \mathcal{N}^+\},
\end{equation*}
and
\begin{equation*}
\sigma:=\inf\{\liminf_{n\to \infty}\Phi_\lambda(u_k):u_k\in \mathcal{M}\},
\end{equation*}
where
\begin{equation*}
\mathcal{M}=\{u_k\in\mathcal{N}:\lim_{n\to \infty}\psi''_{u_k}(1)=0\}.
\end{equation*}
With a simple modification of  \cite[Lemma 3.4]{DN} we can prove:
\begin{lem}\label{naimenineq} There holds
	\begin{equation*}
	\frac{(p-2)^2a^2}{4p(4-p)b}\le \sigma\le  c^0.
	\end{equation*}
\end{lem}
\begin{prop}\label{bifurcation} There exists $\varepsilon>0$ such that for all $\lambda\in(\lambda_0^*-\varepsilon,\lambda_0^*)$ we have that
	\begin{equation*}
	c^+<\frac{(p-2)^2a^2}{4p(4-p)b}.
	\end{equation*}
\end{prop}
\begin{proof}
We claim that $c^+ \to 0$ as $\lambda\uparrow \lambda_0^*$. In fact, let $w\in H_0^1(\Omega)$ be such that $\lambda_0^*=\lambda_0(w)$ (see Proposition \ref{lambdazeroachieved}). Since $\lambda^*<\lambda_0^*$, there exists $\varepsilon>0$ such that if $\lambda\in(\lambda_0^*-\varepsilon,\lambda_0^*)$, then the fiber map $\psi_{\lambda,w}$ satisfies iii) of Proposition \ref{graph psi} and hence $t_\lambda^+(w)w\in \mathcal{N}_\lambda^+$. It follows that
	\begin{equation*}
	0\leq c^+ \le \Phi_\lambda(t_\lambda^+(w)w)\to 0,\ \mbox{as}\ \lambda\uparrow \lambda_0^*.
	\end{equation*}
	To conclude we choose $\varepsilon>0$ in such a way that for each $\lambda\in(\lambda_0^*-\varepsilon,\lambda_0^*)$ we have that $c^+<\frac{(p-2)^2a^2}{4p(4-p)b}$.
\end{proof}

\begin{proof}[Proof of Theorem \ref{T41}] Let $\varepsilon>0$ be given as in Proposition \ref{bifurcation}. With a simple adaptation of the proof of \cite[Corollary 3.3]{DN}, one can use Lemma \ref{naimenineq} and Proposition \ref{bifurcation}, to show the existence of a Palais-Smale sequence $u_k\in \mathcal{N}_\lambda^+$ such that  $\Phi_\lambda(u_k)\rightarrow c^+$. As in the proof of \cite[Proposition 4.2]{DN} we have that $(u_k)_k$ is bounded and $u_k\rightharpoonup u_0$ with $u_0\neq 0$. We claim that $(u_k)_k$ has a strongly convergent subsequence. Indeed, suppose on the contrary and define $\tilde{u}_k:=u_k-u_0$.
	
Let $e>0$ satisfies $\|u_0\|^2-e\lim_{k\to \infty}\|\tilde{u}_k\|^2=0$ and define
\begin{equation*}
w_{k,s}=(1+s)^{\frac{1}{2}}u_0+(1-es)^{\frac{1}{2}}\tilde{u}_k \ \ \ \mbox{and}\ \ \ \overline{h}(s)=\lim_{k\to \infty}\Phi_\lambda(w_{k,s}),\ \forall s\in (-1,1/e),
\end{equation*}
and observe that
\begin{equation}\label{c+1}
\overline{h}(0)=c^+,\  \overline{h}'(0)=0 \ \mbox{and}\ \overline{h}''(0)<0.
\end{equation}
 Define
\begin{equation*}
\overline{g}(s)=\lim_{k\to \infty}\Phi'_\lambda(w_{k,s})w_{k,s},\ \forall s\in (-1,1/e),
\end{equation*}
	and observe (see for details \cite[Proposition 4.2]{DN}) that
	\begin{equation}\label{c+2}
	\overline{g}(0)<0.
	\end{equation}
	From \eqref{c+1} and \eqref{c+2} we deduce that for sufficiently large $k$ we have that  $\Phi_\lambda(w_{k,s})<c^+$ and $\Phi'_\lambda(w_{k,s})w_{k,s}<0$. It follows that $\psi_{\lambda,w_{k,s}}$ satisfies item iii) of Proposition \ref{graph psi 1} and there exists $t_{\lambda}^-(w_{k,s})<1<t_{\lambda}^+(w_{k,s})$ such that $t_{\lambda}^+(w_{k,s})w_{k,s}\in\mathcal{N}_\lambda^+$. Therefore
	\begin{equation*}
	\Phi_\lambda(t_{\lambda}^+(w_{k,s})w_{k,s})<	\Phi_\lambda(w_{k,s})<c^+,
	\end{equation*}
	which is a contradiction. Thus we can assume that $u_k\to u$ in $H_0^1(\Omega)$ and from Proposition \ref{naturalconstraint} the proof is complete.
\end{proof}
\subsection{Existence of the second solution when $a^\frac{N-4}{2}b< C_2(N)$ }
For each $a>0$ and $b,\lambda\ge 0$, define (whenever $\mathcal{N}^-$ is not empty)
\begin{equation*}
c^-:=c^-(a,b,\lambda)=\inf\{\Phi_\lambda(u):u\in \mathcal N^-\}.
\end{equation*}

Now we prove a result which complements \cite[Theorem 1.1]{DN}.
\begin{theor}\label{N-existence} Assume $a^\frac{N-4}{2}b< C_2(N)$. Then, there exists $p_0(a,b)\in(2,2^*)$ such that if $p\in(p_0(a,b),2^*)$,  for all $\lambda>0$, there exists $v_\lambda\in {\mathcal N^-}$ for which $c^-(a,b,\lambda)=\Phi_\lambda(v_\lambda)$.
\end{theor}

\begin{proof} From Proposition \ref{nehariestimatives} in the Appendix we know that
		\begin{equation}\label{Q11}
	c^-(a,b,0)< \frac{(2^*-2)^2a^2}{4\cdot 2^*(4-2^*)b}.
	\end{equation}
		Note that the function $[2,2^*)\ni p\mapsto \frac{(p-2)^2a^2}{4p(4-p)b}$ is increasing and is zero for $p=2$, therefore from \eqref{Q11}, there exists a unique $p_0:=p_0(a,b)\in(2,2^*)$ such that
			\begin{equation*}
		c^-(a,b,0)=\frac{(p_0-2)^2a^2}{4p_0(4-p_0)b}.
		\end{equation*}
		As a consequence
			\begin{equation*}
		c^-(a,b,0)< \frac{(p-2)^2a^2}{4p(4-p)b},
		\end{equation*}
		 for all $p\in(p_0(a,b),2^*)$. From Proposition \ref{decreasingincreasing} and Corollary \ref{nehrilambda>0} in the Appendix and Lemma \ref{naimenineq} we deduce that
		 	\begin{equation*}
		 c^-(a,b,\lambda)\le c^-(a,b,0)< \frac{(p-2)^2a^2}{4p(4-p)b}\le \sigma,  \ \forall \lambda>0
		 \end{equation*}
		 and  from \cite[Corollary 3.3 and Proposition 4.1]{DN}, the proof is complete.
\end{proof}

\begin{rem}\label{comparisonnaimen} Note that:
	\begin{itemize}
		\item[i)] Our method to prove Theorem \ref{N-existence} also proves \cite[Theorem 1.1]{DN}. Indeed, fix $p\in (2,2^*)$. By one hand we know from Proposition \ref{decreasingincreasing} in the Appendix that $c^-(a,b,0)$ is non-decreasing in $b$. On the other hand
		\begin{equation*}
		\lim_{b\downarrow 0}\frac{(p-2)^2a^2}{4p(4-p)b}=\infty,
		\end{equation*}
		therefore by choosing $b$ sufficiently small we conclude that
		\begin{equation*}
		c^-(a,b,\lambda)\le c^-(a,b,0)< \frac{(p-2)^2a^2}{4p(4-p)b},  \ \forall \lambda>0.
		\end{equation*}
		which is \cite[Corollary 3.3]{DN} and consequently implies \cite[Theorem 1.1]{DN}.
				\item[ii)] Observe that the method employed in \cite[Corollary 3.3]{DN}, which was used to prove \cite[Theorem 1.1]{DN}, does not work for all values of $p$ and $a,b>0$ with $a^\frac{N-4}{2}b< C_2(N)$. Indeed, fix $a,b>0$ with $a^\frac{N-4}{2}b< C_2(N)$. Choose $p\in (2,2^*)$ such that
				\begin{equation*}
				\frac{(p-2)^2a^2}{4p(4-p)b}<c^-(a,b,0).
				\end{equation*}
						Therefore from Proposition \ref{continuousA} in Appendix we deduce that for small $\lambda$,
						\begin{equation*}
						\frac{(p-2)^2a^2}{4p(4-p)b}<c^-(a,b,\lambda),
						\end{equation*}
						which contradicts the inequality in \cite[Proposition 3.1]{DN} that was used to prove \cite[Corollary 3.3]{DN}.
	\end{itemize}
\end{rem}
\begin{proof}[Proof of Theorem \ref{T5}] From Theorem \ref{N-existence}, there exists $v_\lambda\in \mathcal{N}^-$ such that $\Phi_\lambda(v_\lambda)=c^-(a,b,\lambda)$. From Proposition \ref{naturalconstraint} the proof is complete.
\end{proof}

However, without any restriction on $p$ or $a,b$, we can prove the following:
\begin{theor}\label{N-existence1} For each $a,b>0$ there exists $\tilde\lambda:=\tilde\lambda(a,b,p)>0$ such that for all $\lambda>\tilde\lambda$, there exists  $v_\lambda\in\mathcal N^- $ satisfying $c^-(a,b,\lambda)=\Phi_\lambda(v_\lambda)$.
\end{theor}
\begin{proof} We claim that $c^-(a,b,\lambda)\to 0$ as $\lambda\to \infty$. To prove it, fix $u\in H_0^1(\Omega)\setminus\{0\}$. Given $\varepsilon>0$, fix any $\lambda'>0$. Then there exists $\delta>0$ such that $0<\psi_{\lambda',u}(t)\le \varepsilon$ for all $t\in(0,\delta]$. Since the function $(\lambda',\infty)\ni\lambda\mapsto \psi_{\lambda,u}(\delta)$ is continuous, decreasing and tends to $-\infty$ as $\lambda\to \infty$, it follows that there exists a unique parameter $\mu>\lambda'$ such that $\psi_{\mu,u}(\delta)=0$. Therefore $0<t_{\mu}^-<\delta$ and $\psi_{\mu,u}(t_{\mu}^-)\le \psi_{\lambda',u}(t_{\mu}^-)\le \varepsilon$. By the arbitrariness of  $\varepsilon$, the claim is proved.

	Now choose $\tilde \lambda$ such that
	\begin{equation*}
	c^-(a,b,\tilde\lambda)< \frac{(p-2)^2a}{4p(4-p)b},
	\end{equation*}
	then from Proposition \ref{decreasingincreasing} in the Appendix we deduce that
		\begin{equation}\label{PP1}
	c^-(a,b,\lambda)< \frac{(p-2)^2a}{4p(4-p)b}, \ \forall \lambda>\tilde\lambda.
	\end{equation}
	Now we divide the proof in two cases: if $a^\frac{N-4}{2}b< C_2(N)$, then we can apply \cite[Corollary 3.3 and Proposition 4.1]{DN} and the proof is complete. Now assume that $a^\frac{N-4}{2}b\ge C_2(N)$. Let $(u_k)_k\in \mathcal{N}^-$ be a minimizing sequence to $c^-(a,b,\lambda)$. Since
	\begin{equation}\label{N-1}
	a\|u_k\|^2+b\|u_k\|^4-\|u_k\|_{2^*}^{2^*}-\lambda\|u_k\|_p^p=0, \forall k,
	\end{equation}
	we deduce that there exist positive constants $d_1,d_2$ such that $d_1\le \|u_k\|\le d_2$ for all $k\in \mathbb N$.  Without loss of generality we can assume that $u_k\rightharpoonup u$ in $H_0^1(\Omega)$ and $\|u_k\|\to t>0$. We claim that $u\neq 0$. Indeed, from \eqref{N-1} and the Sobolev embedding we also have that
	\begin{equation*}
	h(\|u_k\|)= a+b\|u_k\|^2-S_N^{-\frac{2^*}{2}}\|u_k\|^{2^*-2}\le C\lambda\|u_k\|^{p-2},
	\end{equation*}
	where $C$ is some positive constant. Then, if $u= 0$, we would reach the contradiction $0<h(t)\le 0$ (see Proposition \ref{functionsappl} and Lemma \ref{functions}). From Lemma \ref{variational property} we have that
	\begin{equation*}
	\psi_{\lambda,u}'(1)=a\|u\|^2+b\|u\|^4-\|u\|_{2^*}^{2^*}-\lambda\|u\|_p^p\le \liminf_{k\to \infty}(a\|u_k\|^2+b\|u_k\|^4-\|u_k\|_{2^*}^{2^*}-\lambda\|u_k\|_p^p)=0,
	\end{equation*}
	which implies that the fiber map $\psi_{\lambda,u}$ satisfies ii) or iii) of Proposition \ref{graph psi 1}. We claim that it satisfies iii). Indeed, if it satisfies ii), then $u\in \mathcal{N}^0$ and from Lemma \ref{variational property} and \eqref{PP1} we obtain that
		\begin{equation*}
	\Phi_\lambda(u)\le \liminf_{k\to \infty}\Phi_\lambda(u_k)=c^-(a,b,\lambda)< \frac{(p-2)^2a}{4p(4-p)b},
	\end{equation*}
	which contradicts Lemma \ref{naimenineq}. Therefore $\psi_{\lambda,u}$ satisfies iii) and there exists $t_{\lambda}^-(u)\le 1$ such that $ t_{\lambda}^-(u)u\in \mathcal{N}^-$. From Lemma \ref{variational property}
	\begin{equation*}
	\Phi_\lambda( t_{\lambda}^-(u)u)\le \liminf_{k\to \infty}\Phi_\lambda( t_{\lambda}^-(u)u_k)\le \liminf_{k\to \infty}\Phi_\lambda(u_k)=c^-(a,b,\lambda),
	\end{equation*}
	and the proof is complete.
\end{proof}

\begin{rem}  Note that
	\begin{itemize}
		\item[i)]  Theorem \ref{N-existence1} complements the results of \cite{DN}, globally in $a,b$ and locally in $\lambda$.
		\item[ii)] Recall from Theorems \ref{MPS} and \ref{MPS>0} that if $a^\frac{N-4}{2}b> C_2(N)$ and $\lambda>\lambda_0^*-\varepsilon$, then $\Phi_\lambda$ has a mountain pass type solution. One may ask if the solutions found in Theorem \ref{N-existence1} and in those theorems are the same? Or at least, is it true that $c^-(a,b,\lambda)=c_{\lambda}$?
	\end{itemize}
\end{rem}
\begin{proof}[Proof of Theorem \ref{T6}] From Theorem \ref{N-existence1}, there exists $v_\lambda\in \mathcal{N}^-$ such that $\Phi_\lambda(v_\lambda)=c^-(a,b,\lambda)$. From Proposition \ref{naturalconstraint} the proof is complete.
\end{proof}

\subsection{Brezis-Nirenberg problem: the limit case $b\to 0$}
In this Section we show how to recover a well known result from Brezis and Nirenberg  \cite{BN} as a byproduct of our study. To emphasize the more important role of the parameter $b$, we use the notation $\psi_{b,\lambda,u}=\psi_{\lambda,u}$, $t_{b,\lambda}^-(u)=t_\lambda^-(u)$, $\Phi_{b,\lambda}=\Phi_\lambda$ and so on.
\begin{lem} Fix $a>0$, then
	\begin{equation*}
	c^-(a,0,0)=\frac{a^{\frac{N}{2}}}{N}S_N^{\frac{N}{2}}.
	\end{equation*}
\end{lem}
\begin{proof} Indeed, first observe that
	\begin{equation*}
	\Phi_{0,0}(u)=\frac{1}{N}a\|u\|^2, \ \forall u\in \mathcal{N}_{0,0}^-,
	\end{equation*}
	which implies from the definition of $S_N$ that
	\begin{equation*}
	\Phi_{0,0}(u)\ge \frac{a^\frac{N}{2}}{N}S_N^{\frac{N}{2}}, \ \forall u\in \mathcal{N}_{0,0}^-.
	\end{equation*}
	Now suppose that $(u_k)_k$ is a minimizing sequence to $S_N$ satisfying $\|u_k\|_{2^*}=1$ for all $k\in\mathbb N$. From Lemma \ref{graph psi 2} and Remark \ref{globalmax} in Appendix, for each $k$, there exists $t_k:=t_{0,0}(u_k)$ such that $t_ku_k\in  \mathcal{N}_{0,0}^-$. From
	\begin{equation*}
	at_k^2\|u_k\|^2-t_k^{2^*}\|u_k\|_{2^*}^{2^*}=0,
	\end{equation*}
	we have that
	\begin{equation*}
	t_k\to \left(a S_N\right)^{\frac{1}{2^*-2}},\ k\to \infty.
	\end{equation*}
	Therefore
		\begin{equation*}
	\Phi_{0,0}(t_ku_k)=\frac{1}{N}at_k^2\|u_k\|^2\to \frac{1}{N}a \left(a S_N\right)^{\frac{2}{2^*-2}}S_N= \frac{a^{\frac{N}{2}}}{N}S_N^{\frac{N}{2}},
	\end{equation*}
	and the proof is complete.
\end{proof}

\begin{prop}\label{energylambdanega} Fix $a>0$, then for each $\lambda>0$ we have that
	\begin{equation*}
	c^-(a,0,\lambda)<c^-(a,0,0)=\frac{a^{\frac{N}{2}}}{N}S_N^{\frac{N}{2}}.
	\end{equation*}
\end{prop}
\begin{proof} For each $\varepsilon>0$, choose $u_{\varepsilon}\in H_0^1(\Omega)$ such that (see \cite{BN})
	\begin{equation*}
	\int_\Omega |\nabla u_\varepsilon|^2=1,\ \ \int_\Omega |u_\varepsilon|^{2^*}=S_N^{\frac{-2^*}{2}}+O(\varepsilon^{\frac{2^*N}{4}}),\ \ \int_\Omega |u_\varepsilon|^p=\frac{\varepsilon^{\frac{2p-N(p-2)}{4}}}{(c+O(1)\varepsilon^{\frac{N-2}{2}})^{\frac{p}{2}}},
	\end{equation*}
	where $c$ is a positive constant independent on $\varepsilon$. From Lemma \ref{graph psi 2} and Remark \ref{globalmax} in Appendix, for each $\varepsilon>0$, there exists $t_{\varepsilon,\lambda}:=t_{0,\lambda}^-(u_\varepsilon)$ such that $t_{\varepsilon,\lambda} u_\varepsilon\in\mathcal N^-$. Denote $f_{\varepsilon}(\lambda)=\psi_{0,\lambda}(t_{\varepsilon,\lambda}u_\varepsilon)=\Phi_{0,\lambda}(t_{\varepsilon,\lambda}u_\varepsilon)$. From Lemma \ref{increas} (and its proof) we know that
	\begin{equation*}
	f_{\varepsilon}(\lambda)-f_{\varepsilon}(0)=f'_{\varepsilon}(\theta)\lambda=-\frac{t_{\varepsilon,\theta}^p}{p}\lambda\|u_\varepsilon\|^p_p,
	\end{equation*}
	and hence
	\begin{equation}\label{E1}
		f_{\varepsilon}(\lambda)=f_{\varepsilon}(0)-\frac{t_{\varepsilon,\theta}^p}{p}\lambda\|u_\varepsilon\|^p_p,  \ \forall \varepsilon,
	\end{equation}
	where $\theta:=\theta_{\varepsilon}\in (0,\lambda)$. Now some calculations are in order: note from
	\begin{equation*}
	at_{\varepsilon,\theta}^2=t_{\varepsilon,\theta}^{2^*}\|u_\varepsilon\|_{2^*}^{2^*}+\lambda t_{\varepsilon,\theta}^p\|u_\varepsilon\|_p^p, \ \forall \varepsilon,
	\end{equation*}
	that there exists a positive constant $c_1$ such that
	\begin{equation}\label{E2}
	t_{\varepsilon,\theta}\ge c_1,  \ \forall \varepsilon.
	\end{equation}
	Moreover, since
		\begin{equation*}
	at_{\varepsilon,0}^2-t_{\varepsilon,0}^{2^*}\|u_\varepsilon\|_{2^*}^{2^*}=0,  \ \forall \varepsilon,
	\end{equation*}
	we conclude that
	\begin{equation*}
t_{\varepsilon,0}=\left(\frac{a}{S_N^{\frac{-2^*}{2}}+O(\varepsilon^{\frac{2^*N}{4}})}\right)^{\frac{1}{2^*-2}}=\left(\frac{a}{S_N^{\frac{-2^*}{2}}}\right)^{\frac{1}{2^*-2}}+O(\varepsilon^{\frac{2^*N}{4(2^*-2)}}), \  \forall \varepsilon
	\end{equation*}
	and hence
	\begin{eqnarray}\label{E3}
	f_{\varepsilon}(0)&=& \frac{a}{2}t_{\varepsilon,0}^2-\frac{t_{\varepsilon,0}^{2^*}}{2^*}\|u_\varepsilon\|_{2^*} ^{2^*} \nonumber \\
	&=&  \frac{a}{2}\left[\left(\frac{a}{S_N^{\frac{-2^*}{2}}}\right)^{\frac{2}{2^*-2}}+O(\varepsilon^{\frac{2^*N}{2(2^*-2)}})\right]-\frac{1}{2^*}\left[\left(\frac{a}{S_N^{\frac{-2^*}{2}}}\right)^{\frac{2^*}{2^*-2}}+O(\varepsilon^{\frac{2^*2^*N}{4(2^*-2)}})\right]\left(S_N^{\frac{-2^*}{2}}+O(\varepsilon^{\frac{2^*N}{4}})\right), \nonumber \\
	&=& \frac{a^{\frac{N}{2}}}{N}S_N^{\frac{N}{2}}+O(\varepsilon^{\frac{2^*N}{4}}).
	\end{eqnarray}
	We combine \eqref{E1} and \eqref{E3} to obtain that
	\begin{eqnarray*}
		f_{\varepsilon}(\lambda)&=&\frac{a^{\frac{N}{2}}}{N}S_N^{\frac{N}{2}}+O(\varepsilon^{\frac{2^*N}{4}})-\frac{t_{\varepsilon,\theta}^p}{p}\lambda\frac{\varepsilon^{\frac{2p-N(p-2)}{4}}}{(c+O(1)\varepsilon^{\frac{N-2}{2}})^{\frac{p}{2}}}, \\
		&=& \frac{a^{\frac{N}{2}}}{N}S_N^{\frac{N}{2}}+\varepsilon^{\frac{2p-N(p-2)}{4}}\left[\frac{O(\varepsilon^{\frac{2^*N}{4}})}{\varepsilon^{\frac{2p-N(p-2)}{4}}}-\frac{t_{\varepsilon,\theta}^p}{p}\lambda\frac{1}{(c+O(1)\varepsilon^{\frac{N-2}{2}})^{\frac{p}{2}}}\right].
	\end{eqnarray*}
	Since
\begin{equation*}
\frac{2^*N}{4}>1>\frac{2p-N(p-2)}{4},
\end{equation*}
we conclude from \eqref{E2} that for sufficiently small $\varepsilon$, we must have that $f_\varepsilon(\lambda)<\frac{a^{\frac{N}{2}}}{N}S_N^{\frac{N}{2}}$ which concludes the proof.
	\end{proof}
\begin{rem}\label{bstar} Fix $a>0$ and $\lambda\ge 0$:
	\begin{itemize}
		\item[i)] By using a continuity argument, one can easily see that the Nehari manifold $\mathcal{N}_{b,\lambda}^-$ is not empty for $b$ on a neighborhood of $0$.
		\item[ii)] However, it is possible to adapt the calculations made in Theorem \ref{nonexistence}, to prove the existence of $b^*>0$ such that if $b\in [0,b^*)$, then $\mathcal{N}_{b,\lambda}^-\neq\emptyset$, while if $b>b^*$, then $\mathcal{N}_{b,\lambda}=\emptyset$ (see Appendix \ref{AB}).
\end{itemize}
\end{rem}
	As a corollary of Theorem \ref{N-existence1} we obtain the following result \`{a} la Brezis Nirenberg \cite{BN}:
\begin{theor}\label{BNR} Let $a=1$ and $b_k\downarrow 0$. Then, there exists a sequence $(v_k)_k$ of solutions of $(\mathcal P_\lambda)$ such that $v_k\to v$ where $v$ is a nontrivial solution of 
	
	\begin{equation*}	(\mathcal Q_\lambda) \ \ \ \ \left\{
	\begin{array}{ll}
	- \Delta u=
	|u|^{2^*-2}u+\lambda |u|^{p-2}u, & \hbox{ in } \Omega  \\
	u=0, & \hbox{on } \partial \Omega.
	\end{array}
	\right.
	\end{equation*}
\end{theor}
\begin{proof} Fix $\lambda>0$.  From Remark \ref{bstar} we can assume that $c^-(1,b_k,\lambda)$ is well defined for all $k$. Let also $\varepsilon>0$ such that $c^-(1,0,\lambda)+ \varepsilon< \frac{1}{N}S_N^{\frac{N}{2}}$ (see Proposition \ref{energylambdanega}). Thus, by Proposition \ref{continuousA} in Appendix, for $k$ big enough, one has
	\begin{equation*}
	c^-(1,b_k,\lambda)< c^-(1,0,\lambda)+\varepsilon
	< \frac{1}{N}S_N^{\frac{N}{2}}.	
	\end{equation*}
	 We claim that $(v_k)_k$ is bounded in $H^1_0(\Omega)$.
	Indeed, we know that
	\begin{equation}\label{BN1}
	0=\Phi_{b_k,\lambda}'(v_k)(v_k)=\|v_k\|^2+{b_k} \|v_k\|^4-\|v_k\|^{2^*}_{2^*}-{\lambda}\|v_k\|_p^p
	\end{equation}
	\begin{equation}\label{BN2}c^-(1,b_k,\lambda)=\Phi_{b_k,\lambda}(v_k)=\frac{1}{2}\|v_k\|^2+\frac{b_k}{4} \|v_k\|^4-\frac{1}{2^*}\|v_k\|^{2^*}_{2^*}-\frac{\lambda}{p}\|v_k\|_p^p.
	\end{equation}
	Denote $t_k=t_{0,\lambda}^-(v_k)$ and note from Lemma \ref{increas} in the Appendix that $0< t_k \le 1$ for all $k$. This property combined with Proposition \ref{continuousA} implies that
	\begin{eqnarray*}
		0&\le& \lim_{k\to \infty}\Phi_{0,\lambda}(t_kv_k)-c^-(1,0,\lambda), \\
		&\le& \lim_{k\to \infty} \Phi_{b_k,\lambda}(t_kv_k)-c^-(1,0,\lambda), \\
		&\le & \lim_{k\to \infty} \Phi_{b_k,\lambda}(v_k)-c^-(1,0,\lambda), \\
		&=& \lim_{k\to \infty} \left(\Phi_{b_k,\lambda}(v_k)-c^-(1,b_k,\lambda)\right)=0,
	\end{eqnarray*}
	and hence $(t_kv_k)_k$ is a minimizing sequence to $c^-(1,0,\lambda)$. We claim that $(t_k)_k$ is bounded away from $0$. Suppose on the contrary that $t_k\to 0$ as $k\to \infty$. Since $t_kv_k\in \mathcal{N}_{0,\lambda}^-$ we know that
	\begin{equation*}
	2t_k^2\|v_k\|^2-2^*t_k^{2^*}\|v_k\|^{2^*}_{2^*}-p{\lambda}t_k^p\|v_k\|_p^p<0, \quad  \forall k.
	\end{equation*}
	Thus
	\begin{equation*}
	2\frac{\|v_k\|^2}{\|v_k\|^{2^*}_{2^*}}-2^*t_k^{2^*-2}<p{\lambda}t_k^{p-2}\frac{\|v_k\|_p^p}{\|v_k\|^{2^*}_{2^*}}, \quad \forall k,
	\end{equation*}
	and hence
	\begin{equation}\label{O1}
	\frac{\|v_k\|^2}{\|v_k\|^{2^*}_{2^*}}=o(1).
	\end{equation}
	From
	\begin{equation*}
	\|v_k\|^2+b_k\|v_k\|^4-\|v_k\|^{2^*}_{2^*}-{\lambda}\|v_k\|_p^p=0, \quad \forall k,
	\end{equation*}
	and \eqref{O1} we deduce that
	\begin{equation} \label{O2}
	\frac{b_k\|v_k\|^4}{\|v_k\|_{2^*}^{2^*}}=1+\lambda\frac{\|v_k\|_p^p}{\|v_k\|^{2^*}_{2^*}}+o(1), \quad \forall k.
	\end{equation}
	Since
	\begin{eqnarray*}
		\Phi_{b_k,\lambda}(v_k)=\|v_k\|^{2^*}_{2^*}\left(\frac{1}{2}\frac{\|v_k\|^2}{\|v_k\|^{2^*}_{2^*}}+\frac{1}{4}\frac{b_k\|v_k\|^4}{\|v_k\|_{2^*}^{2^*}}-\frac{1}{2^*}-\frac{\lambda}{p}\frac{\|v_k\|_p^p}{\|v_k\|^{2^*}_{2^*}}\right), \quad \forall k,
	\end{eqnarray*}
	it follows from \eqref{O2} that
	\begin{eqnarray*}
		\Phi_{b_k,\lambda}(v_k)&=&\|v_k\|^{2^*}_{2^*}\left[\frac{1}{4}\left(1+\lambda\frac{\|v_k\|_p^p}{\|v_k\|^{2^*}_{2^*}}\right)-\frac{1}{2^*}-\frac{\lambda}{p}\frac{\|v_k\|_p^p}{\|v_k\|^{2^*}_{2^*}}+o(1)\right] \\
		&=& \|v_k\|^{2^*}_{2^*}\left[\frac{2^*-4}{2^*4}+\left(\frac{p-4}{2^*4}\right)\lambda\frac{\|v_k\|_p^p}{\|v_k\|^{2^*}_{2^*}}+o(1)\right],
	\end{eqnarray*}
	which is a contradiction since $\Phi_{b_k,\lambda}(v_k)=c^-(1,b_k,\lambda)>0$ for all $k$ and therefore $t_k$ is bounded away from $0$. Once $(t_kv_k)_k$ is a minimizing sequence to $c^-(1,0,\lambda)$, it has to be bounded, that is, there exists $d>0$ such that
	\begin{equation*}
	t_k^2\int|\nabla v_k|^2\le d, \quad \forall k,
	\end{equation*}
	and as a consequence $(v_k)_k$ is bounded in $H_0^1(\Omega)$.

	Eventually passing to a subsequence, there exists $v\in H^1_0(\Omega)$ such that
	$v_k\rightharpoonup v$ weakly in $H^1_0(\Omega)$,  $v_k\rightarrow v$ strongly in $L^q(\Omega)$ for $q<2^*$,
	$|v_k|^{2^*-2} v_k\rightharpoonup |v|^{2^*-2}v$ weakly in $(L^{2^*})'$. Thus, since  $v_k$ is a critical point of $\Phi_{k,\lambda}$, for every $\varphi\in H^1_0(\Omega)$,
	\[(1+b_k\|v_k\|^2)\int_{\Omega}\nabla v_k\nabla \varphi -\int_{\Omega }|v_k|^{2^*-2}v_k\varphi-\lambda\int_{\Omega}|v_k|^{p-2}v_k\varphi =0,  \]  passing to the limit as $k\to\infty$ we deduce that
	\[\int_{\Omega}\nabla v\nabla \varphi -\int_{\Omega }|v|^{2^*-2}v\varphi-\lambda\int_{\Omega}|v|^{p-2}v\varphi =0,  \] which implies that $v$ is a solution of $(\mathcal Q_\lambda)$.  Let us show that $v\neq 0$. Assume by contradiction that $v=0$. By \eqref{BN1}, dividing by $\|v_k\|^2$ we get
	\[
	1+b_k\|v_k\|^2-S_N^{-\frac{2}{2^*}}\|v_k\|^{2^*-2}\leq
	1+b_k\|v_k\|^2-\|v_k\|_{2^*}^{2^*-2}=\lambda \|v_k\|_{p}^{p-2}
	\leq c_1\lambda \|v_k\|^{p-2}	\] and $(\|v_k\|)_k$ is bounded away from zero. Passing to a subsequence we  can assume that $\|v_k\|\to l>0$. From \eqref{BN1} and \eqref{BN2} (recall that $0=v=\lim_kv_k$ in $L^p$),  we obtain that
	$$l^2=\lim_k\|v_k\|_{2^*}^{2^*}$$  and $$\lim_k c^-(1,b_k,\lambda)=\frac{1}{2}l^2-\frac{1}{2^*}\lim_k\|v_k\|^{2^*}_{2^*}=\frac{1}{N}l^2.$$
	Since $\|v_k\|^2 \geq S_N \|v_k\|_{2^*}^{2}$
	we obtain that $l^2\geq S_N^{\frac{N}{2}}$ which implies
	\[\lim_k c^-(1,b_k,\lambda)\geq \frac{1}{N} S_N^{\frac{N}{2}},  \] against the initial assumptions. Thus, $v\neq 0.$  Let us prove now that 
	 $v_k\to v$ in $H_0^1(\Omega)$ and $\Phi_{0,\lambda}(v)=c^-(1,0,\lambda)$. Indeed, since $v_k\in \mathcal{N}_{b_k,\lambda}^-$ for all $k$, we have that
			\begin{equation*}
			\Phi_{b_k,\lambda}(v_k)=\frac{2^*-2}{22^*}\|v_k\|^2+\frac{2^*-4}{42^*}b_k\|v_k\|^4-\frac{2^*-p}{2^*p}\|v_k\|_p^p,  \ \forall k.
			\end{equation*}
			Since $v$ solves $(\mathcal Q_\lambda)$ we conclude from Remark \ref{globalmax} in the Appendix that $v\in \mathcal{N}_{0,\lambda}^-$ and hence
			\begin{equation*}
			c^-(1,0,\lambda)\leq\Phi_{0,\lambda}(v)=\frac{2^*-2}{22^*}\|v\|^2-\frac{2^*-p}{2^*p}\|v\|_p^p\le \liminf_{k\to \infty}\Phi_{b_k,\lambda}(v_k)=c^-(1,0,\lambda),
			\end{equation*}
			and therefore $\|v_k\|\to \|v\|$ as $k\to \infty$, which implies that $v_k\to v$ in $H_0^1(\Omega)$ and $\Phi_{0,\lambda}(v)=c^-(1,0,\lambda)$.
		\end{proof}
	\begin{proof}[Proof of Theorem \ref{T7}] See  Theorem \ref{BNR}.
	\end{proof}

\appendix

\section{Some topological properties of the Nehari manifolds}\label{A}
We collect some topological properties concerning the Nehari manifold $\mathcal{N}^-$. Since the dependency on each parameter will be considered, we will write the full notation $\Phi_{a,b,\lambda}$, $t_{a,b,\lambda}^-(u)$, $\mathcal{N}_{a,b,\lambda}^-$ and so on.

Similarly to Proposition \ref{graph psi 1} we can prove:
\begin{lem}\label{graph psi 2}
	For each $a>0$, $b\in \mathbb{R}$, $\lambda\in \mathbb{R}$ and $u\in H^1_0(\Omega)\setminus\{0\},$ only one of the next $i)-iv)$ occurs.
	
	\begin{itemize}
		\item[i)] The function $ \psi_{a,b,\lambda,u}$ is increasing and has no critical points.
		\item[ii)] The function $ \psi_{a,b,\lambda,u}$ has only one critical point in $]0,+\infty[$ at the value $t_{a,b,\lambda}(u)$. Moreover, $ \psi''_{a,b,\lambda,u}(t_{a,b,\lambda}(u))=0$ and $ \psi_{a,b,\lambda,u}$ is increasing.
		\item[iii)] The function $ \psi_{a,b,\lambda,u}$ has only two critical points, $0 < t^-_{a,b,\lambda} (u) < t^+_{a,b,\lambda} (u)$. Moreover, $t^-_{a,b,\lambda}(u)$ is a local maximum and  $t^+_{a,b,\lambda}(u)$ is a local minimum with $ \psi_{a,b,\lambda,u}''(t^-_{a,b,\lambda}(u))<0<\psi_{a,b,\lambda,u}''(t^+_{a,b,\lambda}(u))$.
		\item[iv)] The function $ \psi_{a,b,\lambda,u}$ has only one critical point in $]0,+\infty[$ at the value $t^-_{a,b,\lambda}(u)$.  Moreover, $t^-_{a,b,\lambda}(u)$ is a local maximum and $ \psi_{a,b,\lambda,u}''(t^-_{a,b,\lambda}(u))<0$.
 	\end{itemize}
\end{lem}
\begin{rem}\label{globalmax} If $b\le 0$ and $\lambda \ge 0$, then only item iv) of Lemma \ref{graph psi 2} occurs. Moreover, if $b>0$, then only one of the items $i)-iii)$ occurs.
\end{rem}

\begin{lem}\label{increas} Fix $u\in H_0^1(\Omega)\setminus\{0\}$ and $a>0$. Let $V\subset \mathbb{R}^2$ be an open set and assume that $t_{a,b,\lambda}^-(u)$ is defined for all $(b,\lambda)\in V$. Then the function $V\ni (b,\lambda)\mapsto t_{a,b,\lambda}^-(u)$ is $C^1$. Moreover the following holds true.
	\begin{itemize}
		\item[i)]  The functions $t_{a,b,\lambda}^-(u)$ and  $\psi_{a,b,\lambda,u}(t_{a,b,\lambda}^-(u))$ are increasing with respect to $b$;
		\item[ii)] The functions $t_{a,b,\lambda}^-(u)$ and  $\psi_{a,b,\lambda,u}(t_{a,b,\lambda}^-(u))$ are decreasing with respect to $\lambda$.
	\end{itemize}
\end{lem}
\begin{proof} Denote $t_{b,\lambda}=t_{a,b,\lambda}^-(u)$ and note from the implicit function theorem that $\psi_{a,b,\lambda,u}'(t_{b,\lambda})=0$ and $\psi_{a,b,\lambda,u}''(t_{b,\lambda})<0$ implies that $t_{b,\lambda}$ is $C^1$ as a function of $(b,\mu,\lambda)\in V$. Since
	\begin{equation*}
	at_{b,\lambda}^2\|u\|^2+bt_{b,\lambda}^4\|u\|^4-t_{b,\lambda}^{2^*}\mu\|u\|_{2^*}^{2^*}-\lambda t_{b,\lambda}^p\|u\|_p^p=0,
	\end{equation*}
	we conclude by differentiating both sides, with respect to $b$, that
	\begin{equation*}
	\frac{\partial{t_{b,\lambda}}}{\partial b}=-\frac{t_{b,\lambda}^4\|u\|^4}{\psi_{a,b,\lambda,u}''(t_{b,\lambda})}>0,
	\end{equation*}
	and hence $t_{b,\lambda}$ is increasing in $b$. Now let $f(b)=\psi_{a,b,\lambda,u}(t_{a,b,\lambda}^-(u))$ and observe that
	\begin{equation*}
	f'(b)=\frac{\partial{t_{b,\lambda}}}{\partial b}\psi_{a,b,\lambda,u}'(t_{b,\lambda})+\frac{t_{b,\lambda}^4\|u\|^4}{4}>0,
	\end{equation*}
	which implies that $f$ is increasing and hence i) is proved. The proof of ii) is similar.
\end{proof}
\begin{rem}\label{t+} Note that a similar result can also be proved with respect to the functions $t_{a,b,\lambda}^+(u)$ and $\psi_{a,b,\lambda,u}(t_{a,b,\lambda}^+(u))$.
\end{rem}
Denote
\begin{equation*}
\mathcal{M}_{a,b,\lambda}=\left\{\frac{u}{\|u\|}:u\in\mathcal{N}_{a,b,\lambda}^-\right\}.
\end{equation*}
\begin{lem}\label{containedmu} There holds:
	\begin{itemize}
		\item[i)] If $b_1< b_2$, then $\mathcal{M}_{b_2}\subset \mathcal{M}_{b_1}$.
		\item[ii)] If $\lambda_1< \lambda_2$, then $\mathcal{M}_{\lambda_1}\subset \mathcal{M}_{\lambda_2}$.
	\end{itemize}
\end{lem}

\begin{proof} i)  Take $u\in \mathcal{M}_{a,b_2,\lambda}$. Once $\psi_{a,b_1,\lambda}'(t)\le  \psi_{a,b_2,\lambda}'(t)$ for all $t>0$, it follows that $\psi_{a,b_1,\lambda}'(t_{a,b_2,\lambda}^-(u))<  \psi_{a,b_2,\lambda}'(t_{a,b_2,\lambda}^-(u))=0$ and hence, from Lemma \ref{graph psi 2} we conclude that $u\in  \mathcal{M}_{a,b_1,\lambda}$.
	
ii) 	Take $u\in \mathcal{M}_{a,b,\lambda_1}$. Once $\psi_{a,b,\lambda_2}'(t)\le  \psi_{a,b,\lambda_1}'(t)$ for all $t>0$, it follows that $\psi_{a,b,\lambda_2}'(t_{a,b,\lambda_1}^-(u))<
\psi_{a,b,\lambda_1}'(t_{a,b,\lambda_1}^-(u))=0$ and hence, from Proposition \ref{graph psi 2} we conclude that $u\in  \mathcal{M}_{a,b,\lambda_1}$.
\end{proof}
\begin{prop}\label{decreasingincreasing} Fix $a>0$ and let $I$ be an  interval. Then, the following holds true.
	\begin{itemize}
		\item[i)] Fix $b\in \mathbb{R}$. If $c^-(a,b,\lambda)$ is defined for all $\lambda\in I$, then it is non-increasing as a function of $\lambda$.
		\item[ii)] Fix $\lambda\in \mathbb{R}$. If $c^-(a,b,\lambda)$ is defined for all $b\in I$, then it is non-decreasing as a function of $b$.
 	\end{itemize}
\end{prop}
\begin{proof} i) Indeed, fix $ \lambda_1<\lambda_2$ and $u\in \mathcal{M}_{a,b,\lambda_1}$. Since from Lemma \ref{containedmu} we have that $u\in \mathcal{M}_{a,b,\lambda_2}$, it follows from Lemma \ref{increas}  that
	\begin{equation}\label{ineincere}
	c^-(a,b,\lambda_2)\le \psi_{a,b,\lambda_2}(t_{a,b,\lambda_2}^-(u))<\psi_{a,b,\lambda_1}(t_{a,b,\lambda_1}^-(u)), \forall u\in \mathcal{M}_{a,b,\lambda_1}.
	\end{equation}
	and hence $c^-(a,b,\lambda_2)\le c^-(a,b,\lambda_1)$. The proof of ii) is similar.
\end{proof}
\begin{prop}\label{continuousA} Fix $a>0$ and let $I$ be an  interval. Then, the following holds true.
	\begin{itemize}
	\item[i)] Fix $\lambda\in \mathbb{R}$. If $c^-(a,b,\lambda)$ is defined for all $b\in I$, then it is right continuous as a function of $b$.
	\item[ii)] Fix $b\in \mathbb{R}$. If $c^-(a,b,\lambda)$ is defined for all $\lambda\in I$, then it is right continuous as a function of $\lambda$.
\end{itemize}
\end{prop}
\begin{proof} i) Fix $b_0\in I$. We claim that $\lim_{b\downarrow b_0}c^-(a,b,\lambda)=c^-(a,b_0,\lambda)$. Indeed, once $I\ni b\mapsto c^-(a,b,\lambda)$ is non-decreasing, we can assume that $\lim_{b\downarrow b_0}c^-(a,b,\lambda)=c\ge c^-(a,b_0,\lambda)$. Suppose on the contrary that $c>c^-(a,b_0,\lambda)$. Given $\varepsilon>0$ choose $u\in \mathcal{M}_{a,b_0,\lambda}$ such that $\Phi_{a,b_0,\lambda}(t_{a,b_0,\lambda}^-(u)u)\in [c^-(a,b_0,\lambda),c^-(a,b_0,\lambda)+\varepsilon)$ and $c^-(a,b_0,\lambda)+\varepsilon<c$. From Lemma \ref{increas} we conclude that for small $\delta>0$
	\begin{equation*}
	c^-(a,b_0+\delta,\lambda)\le \Phi_{a,b_0+\delta,\lambda}(t_{a,b_0+\delta,\lambda}^-(u)u)< c^-(a,b_0,\lambda)+\varepsilon<c,
	\end{equation*}
	which is a contradiction and thus $I\ni b\mapsto c^-(a,b,\lambda)$ is right continuous. The proof of ii) is similar.
\end{proof}
\section{The case $\lambda=0$}\label{AB}
We collect some results concerning the fiber maps $\psi$ when $\lambda=0$. The parameter now is $b>0$, while $a>0$ is fixed. For this reason, we write $\psi_{b,u}$ and $\Phi_b$ instead of $\psi_{0,u}$ and $\Phi_0$ and so on. As we already know, for each $u\in H_0^1(\Omega)\setminus\{0\}$ the fiber map $\psi_{b,u}$ has satisfies Proposition \ref{graph psi 1}. One can see now that the systems $\psi_{b,u}(t)=\psi_{b,u}'(t)=0$ and $\psi'_{b,u}(t)=\psi_{b,u}''(t)=0$ admits a unique solution, with respect to $t,b$, which are given respectively by  (see \cite{FF} and \cite{S})
\begin{equation*}
t_{0}(u)=\left(\frac{2^*a}{4-2^*}\frac{\|u\|^2}{\|u\|_{2^*}^{2^*}}\right)^{\frac{1}{2^*-2}},
\end{equation*}
\begin{equation*}
b_0(u)=a^{\frac{4-N}{2}}S_N^{\frac{N}{2}}C_1(N)\left(\frac{\|u\|_{2^*}}{\|u\|}\right)^N,
\end{equation*}
and
\begin{equation*}
t(u)=\left(\frac{2a}{4-2^*}\frac{\|u\|^2}{\|u\|_{2^*}^{2^*}}\right)^{\frac{1}{2^*-2}},
\end{equation*}
\begin{equation*}
b(u)=a^{\frac{4-N}{2}}S_N^{\frac{N}{2}}C_2(N)\left(\frac{\|u\|_{2^*}}{\|u\|}\right)^N.
\end{equation*}
As a conclusion of this analysis and similar to Propositions \ref{prop systems} and \ref{prop systems 1} we have 
\begin{prop}\label{prop systems 2} There holds
	
	\begin{itemize}
		\item[i)]
		For each $b\ge b_0(u)$ and each  $u\in H^1_0(\Omega)\setminus \{0\}$, $\inf_{t>0}\psi_{b,u}(t)=0$; for each $b<b_0(u)$ there exists $u\in H^1_0(\Omega)\setminus \{0\}$ such that $\Phi_b(u)<0$.
		\item[ii)] For each $b\ge b(u)$, the set $\mathcal N_b=\emptyset$; for each $b<b(u)$, the sets $\mathcal N_b^+$, $\mathcal N_b^-$ and $\mathcal N_b^0$ are non empty.
	\end{itemize}
\end{prop}
Therefore:
\begin{lem}\label{nelambda0} The following holds true.
	\begin{itemize}
	\item[i)] If $a^\frac{N-4}{2}b<C_1(N)$, then there exists $u\in H_0^1(\Omega)\setminus\{0\}$ such that $\Phi_b(u)<0$.
	\item[ii)] If $a^\frac{N-4}{2}b\ge C_1(N)$, then $\psi_{b	,u}(t)>0$ for all $u\in H_0^1(\Omega)\setminus\{0\}$ and $t>0$.
	\item[iii)] If $a^\frac{N-4}{2}b<C_2(N)$, then $\mathcal{N}^0_b, \mathcal{N}^-_b, \mathcal{N}^+_b$ are non-empty.
	\item[iv)] If $a^\frac{N-4}{2}b\ge C_2(N)$, then $\mathcal{N}_b=\emptyset$.
\end{itemize}
\end{lem}
\begin{rem} Comparing Lemmas \ref{nelambda0} and \ref{variational property} we see that
	\begin{itemize}
		\item[i)] $\Phi_b$ is weak lower semi-continuous if, and only, $\Phi_b(u)\ge 0$ for all $u\in H_0^1(\Omega)$.
		\item[ii)] If $\Phi_b'(u)u>0$ for all $u\in H_0^1(\Omega)\setminus\{0\}$, then $\Phi_b$ satisfies the Palais-Smale condition. Equivalently $\mathcal{N}_b=\emptyset$.
	\end{itemize}
\end{rem}
\begin{cor}\label{nehrilambda>0} If $a^\frac{N-4}{2}b<C_2(N)$, then for all $\lambda>0$ we have $\mathcal{N}_{b}^-\neq\emptyset$.
\end{cor}
\begin{proof} Indeed, this is a consequence of Lemmas \ref{nelambda0} and \ref{containedmu}. We also refer the reader to \cite[Lemma 2.6]{DN}.
\end{proof}
The next lemma is an application of Lemma \ref{increas} and Remark \ref{t+}:
\begin{lem}\label{decreasing} Fix $u\in H_0^1(\Omega)\setminus\{0\}$. The following  holds true.
	\begin{itemize}
		\item[i)] The function $(0,b(u))\ni\mapsto t_b^-(u)$ is continuous and increasing.
		\item[ii)] The function $(0,b(u))\ni\mapsto t_b^+(u)$ is continuous and decreasing.
		\item[iii)]
		\begin{equation*}
		\lim_{b\uparrow b(u)}t_b^-(u)=t(u)=\lim_{b\uparrow b(u)}t_b^+(u).
		\end{equation*}
	\end{itemize}
\end{lem}
The following proposition can be found in \cite{S,S1} (with some adaptations). We give an outline of the proof (recall from Lemma \ref{nelambda0} that $\mathcal{N}^0_b, \mathcal{N}^-_b$ are not empty for all $a,b>0$ satisfying $a^\frac{N-4}{2}b<C_2(N)$):
\begin{prop}\label{nehariestimatives} Suppose that $a^\frac{N-4}{2}b<C_2(N)$, then
	\begin{equation*}
	\Phi_b(u)=\frac{(2^*-2)^2a^2}{4\cdot 2^*(4-2^*)b}, \forall u\in \mathcal{N}^0_b.
	\end{equation*}
	Moreover,
	\begin{equation*}
	c^-(a,b,0)< \frac{(2^*-2)^2a^2}{4\cdot 2^*(4-2^*)b}=c^0(a,b,0).
	\end{equation*}
\end{prop}
\begin{proof} The first part is trivial. Now suppose on the contrary that there exists $u\in\mathcal{N}^-_b$ such that
	\begin{equation*}
	\Phi_b(u)\ge\frac{(2^*-2)^2a^2}{4\cdot 2^*(4-2^*)b}.
	\end{equation*}
	From Lemma \ref{decreasing} we have that $t^-_b(u)=1<t^-_{b'}(u)<t^+_{b'}(u)<t^+_{b}(u)$ for each $0<b<b'<b(u)$ and hence
	\begin{eqnarray*}
		\Phi_{b'}(t^-_{b'}(u)u)&>& \Phi_{b'}(u) \nonumber\\
		&>& \Phi_{b}(u) \nonumber\\
		&\ge& \frac{(2^*-2)^2a^2}{4\cdot 2^*(4-2^*)b}, \nonumber
	\end{eqnarray*}
	which implies that
	\begin{equation*}
\frac{(2^*-2)^2a^2}{4\cdot 2^*(4-2^*)b}<\lim_{b'\uparrow b(u)}	\Phi_{b'}(t^-_{b'}(u)u)=\Phi_{b(u)}(t_b(u)u)=\frac{(2^*-2)^2a^2}{4\cdot 2^*(4-2^*)b(u)},
	\end{equation*}
	a contradiction since $b<b(u)$.
\end{proof}

{\bf Acknowledgments}
F. Faraci has been supported by the Universit\`{a} degli Studi di Catania, "Piano della Ricerca 2016/2018 Linea di intervento 2". She is   member of the Gruppo Nazionale per l'Analisi Matematica, la Probabilit\`{a}
e le loro Applicazioni (GNAMPA) of the Istituto Nazionale di Alta Matematica (INdAM). K. Silva has been supported by CNPq-Grant 408604/2018-2.


\begin{thebibliography}{}
	\bibitem{ACF}
	 C.O. ALVES, F.J. CORR\^{E}A, G.M. FIGUEIREDO, {\it  On a class of nonlocal elliptic problems with critical growth}. Differ. Equ. Appl. {\bf 2} (2010) 409--417.
	\bibitem {BN}
	H. BR\'{E}ZIS, L. NIRENBERG, {\it Positive solutions of nonlinear elliptic equations involving critical Sobolev exponents.} Comm. Pure Appl. Math. {\bf 36} (1983) 437--477.
	\bibitem{CF}
	 F.J. CORR\^{E}A, G.M. FIGUEIREDO, {\it On an elliptic equation of p-Kirchhoff type via variational methods}, Bull. Austral. Math. Soc. {\bf 74} (2006) 263--277.
	 \bibitem{Fan}
	  H. FAN, {\it  Multiple positive solutions for a class of Kirchhoff type problems involving critical Sobolev exponents}. J. Math. Anal. Appl. 431 (2015) 150--168.
	\bibitem{FF}
	F. FARACI, Cs. FARKAS, {\it On an open question of Ricceri concerning a Kirchhoff-type problem}, Minimax Theory and its Applications,  4 (2019) 271--280.
	\bibitem{F}
	G.M. FIGUEIREDO, {\it Existence of a positive solution for a Kirchhoff problem type with critical growth via truncation argument}, J. Math. Anal. Appl. {\bf 401} (2013) 706–713.
	\bibitem{FS}
	 G.M. FIGUEIREDO, J.R. SANTOS, {\it Multiplicity of solutions for a Kirchhoff equation with subcritical or critical growth}. Differential Integral Equations 25 (2012) 853--868.
	 \bibitem{H1}
	 E. HEBEY, {\it  Compactness and the Palais-Smale property for critical Kirchhoff equations in closed manifolds}, Pacific J. Math. {\bf 280} (2016) 41--50.
	 \bibitem{H2}
	  E. HEBEY, {\it  Multiplicity of solutions for critical Kirchhoff type equations},  Comm. Partial Differential Equations {\bf 41} (2016) 913--924.
	\bibitem{Y}
	Y. IL'YASOV, {\it On extreme values of Nehari manifold method via nonlinear Rayleigh's quotient.} Topol. Methods Nonlinear Anal. 49 (2017) 683--714.
	\bibitem{K}
	G. KIRCHHOFF, Mechanik, Teubner, Leipzig, 1883.
	\bibitem{N1} 
	D. NAIMEN, {\it Positive solutions of Kirchhoff type elliptic equations involving a critical Sobolev exponent}, NoDEA Nonlinear Differential Equations Appl. {\bf 21} (2014), No. 6, 885–-914.
\bibitem{DN}
D. NAIMEN, M. SHIBATA, {\it Two positive solutions for the Kirchhoff type elliptic problem with critical nonlinearity in high dimension}, Nonlinear Anal., 186 (2019) 187--208.
\bibitem{N1}
Z. NEHARI, {\it On a class of nonlinear second-order differential equations}, Trans. Amer. Math. Soc. {\bf 95} (1960), 101--123.
\bibitem{N2}
Z. NEHARI, {\it Characteristic values associated with a class of non-linear second-order differential equations}, Acta Math. {\bf 105} (1961), 141–175.
\bibitem{P}
S. I. POKHOZHAEV, {\it The fibration method for solving nonlinear boundary value problems}, Translated in Proc. Steklov Inst. Math. 1992, no. 3, 157--173. Differential equations and function spaces (Russian). Trudy Mat. Inst. Steklov. 192 (1990), 146--163.
\bibitem{PR}
P. PUCCI, V. D. RADULESCU, {\it Progress in nonlinear Kirchhoff problems}, [Editorial]. Nonlinear Anal. {\bf 186}, (2019) 1--5. 	
\bibitem{S}
	K. SILVA,{\it The bifurcation diagram of an elliptic Kirchhoff-type equation with respect to the stiffness of the material}, Z. Angew. Math. Phys. 70 (2019) 13 pp.
	\bibitem{S1}
	K. SILVA,{\it On an abstract bifurcation result concerning homogeneous potential operators with applications to PDEs},
	J. Differential Equations. (in press)
	\bibitem{YM}
	 X. YAO, C. MU, {\it Multiplicity of solutions for Kirchhoff type equations involving critical Sobolev exponents in high dimension},
	Math. Methods Appl. Sci. {\bf 39} (2016) 3722--3734.
	
\end{thebibliography}
\end{document}